\documentclass[a4paper,11pt,oneside]{article}

\makeindex
\usepackage{graphicx}
\usepackage{enumerate}
\usepackage{todonotes}

\usepackage[toc,page]{appendix}
\usepackage{subcaption,siunitx} 
\usepackage[shortlabels]{enumitem}

\usepackage{mathtools}
\usepackage{amsmath}
\usepackage{amsthm}
\usepackage{amssymb}
\usepackage{authblk}
\usepackage{enumerate}
\usepackage{fullpage}
\usepackage{esint}
\usepackage{dsfont} 
\usepackage{colonequals} 

\usepackage{pgfplots}
\pgfplotsset{compat=1.15}
\usepackage{mathrsfs}
\usetikzlibrary{arrows}

\usepackage{hyperref}
\hypersetup{colorlinks, linkcolor=blue, urlcolor=red, filecolor=green, citecolor=blue}

\theoremstyle{plain}
\newtheorem{theorem}{Theorem}[section]

\newtheorem{lemma}[theorem]{Lemma}
\newtheorem{proposition}[theorem]{Proposition}

\newtheorem{definition}[theorem]{Definition}

\theoremstyle{remark}
\newtheorem{remark}[theorem]{Remark}
\theoremstyle{definition}

\newcommand{\N}{\mathbb{N}}
\newcommand{\R}{\mathbb{R}}

\newcommand{\Z}{\mathbb{Z}}
\newcommand\e{\epsilon}

\newcommand\dist{\operatorname{dist}}
\newcommand\sym{\operatorname{sym}}

\newcommand{\SO}[1]{\operatorname{SO}(#1)}
\newcommand\ho{{\operatorname{hom}}}
\newcommand\ext{{\operatorname{ext}}}
\newcommand\per{{\operatorname{per}}}
\newcommand\loc{{\operatorname{loc}}}

\newcommand\II{{\operatorname{I\!I}}}
\newcommand{\step}[1]{\medskip\noindent\textit{Step #1. }}

\newcommand{\placeholder}{\makebox[1ex]{$\boldsymbol{\cdot}$}}
\newcommand{\eps}{\epsilon}
\newcommand{\dd}{\mathrm{d}}
\def\HH{{\mathbf H}}
\def\eff{{\mathrm{ eff}}}
\def\iso{{\mathrm{ iso}}}

\usepackage{cleveref} 

\newcommand\deform{v}

\def\HH{{\mathbf H}}
\def\eff{{\rm eff}}
\def\iso{{\rm iso}}

\begin{document}

\begin{center}
  \LARGE
  Commutativity and non-commutativity of limits in the nonlinear bending theory for prestrained microheterogeneous plates
  \bigskip

  \normalsize
    Klaus B{\"o}hnlein\footnote{klaus.boehnlein@tu-dresden.de}\textsuperscript{A}, 
    Lucas Bouck\footnote{lbouck@andrew.cmu.edu}\textsuperscript{B}, 
  Stefan Neukamm\footnote{stefan.neukamm@tu-dresden.de}\textsuperscript{A}, 
  David Padilla-Garza\footnote{david.padilla-garza@mail.huji.ac.il}\textsuperscript{C} and
  Kai Richter\footnote{kai.richter@tu-dresden.de}\textsuperscript{A} \par \bigskip

  \textsuperscript{A}Faculty of Mathematics, TU Dresden\par \bigskip \textsuperscript{B}Department of Mathematical Sciences at Carnegie Mellon University \par \bigskip \textsuperscript{C}Einstein Institute of Mathematics, Hebrew University of Jerusalem\par \bigskip

  \today
\end{center}

\begin{abstract}
In this paper we study the derivation of nonlinear bending models for prestrained elastic plates from three-dimensional non-linear elasticity via homogenization and dimension reduction. We compare effective models obtained by either simultaneously or consecutively passing to the $\Gamma$-limits as the thickness $h\ll1$ and the size of the material microstructure $\e\ll1$ vanish. In the regime $\e\ll h$  we show that the consecutive and simultaneous limit are equivalent, and also analyze the rate of convergence. In contrast, we observe that there are several different limit models in the case $h\ll \e$.
\end{abstract}

\tableofcontents

\section{Introduction and main result }
In this paper we study the derivation of homogenized bending models from 3d nonlinear elasticity for prestrained, composite plates with a periodic microstructure. The starting point of the derivation is a functional $\mathcal E_{h,\eps}$ describing the elastic energy of a prestrained, three-dimensional plate with thickness $h$ and with a periodic microstructure on scale $\eps$. We are interested in effective models that describe the asymptotic behavior in the case when $h,\eps\ll 1$ and thus study the $\Gamma$-limit of $\mathcal E_{h,\eps}$. Here the limit $h\to 0$ corresponds to dimension reduction leading to a nonlinear bending model, and $\eps\to 0$ corresponds to homogenization. In addition to the purely mathematical interest in simultaneous homogenization and dimension reduction, the analysis of thin periodic sheets is directly motivated by applications \cite{griso2020homogenization,griso2021asymptotic,griso2020asymptotica,griso2020asymptotic,orlik2023asymptotic,falconi2024asymptotic}.

Due to the presence of two small parameters, various different ways to pass to the limit can be considered. In particular, there are two consecutive limits corresponding to dimension reduction after homogenization and vice versa, and simultaneous limits when 
\begin{equation*}
	(h,\eps)\to 0\qquad\text{with }\frac{h}{\eps}\to \gamma\in[0,\infty].
\end{equation*}
Here, $\gamma$ is a scaling parameter that describes the relative scaling between $h$ and $\eps$. 
As indicated in Figure~\ref{Fig:Diagram} and discussed in more detail below, most of the limits have already been established in the literature \cite{neukamm2025linearization,baia2007limit,bohnlein2023homogenized,FJM02,neukamm2015homogenization}. In this paper, we will consider these limits for prestrained materials. The analysis of such materials lies at the intersection of calculus of variations, physics, and geometry (see \cite{kupferman2014riemannian1,kupferman2014riemannian2,lewicka2020dimension,lewicka2014models,jimenez2021dimension}  and \cite{lewicka2023calculus} for a comprehensive review). In the case of the nonlinear bending theory for plates, as shown in \cite{bohnlein2023homogenized} the simultaneous limit leads to an energy functional $\mathcal E^\gamma_{\hom}$ and a spontaneous curvature $B^\gamma_{\rm eff}$ that explicitly depend on the scaling parameter $\gamma$. 

In our paper, we investigate the convergence of $\mathcal E^\gamma_{\hom}$ as $\gamma\to \infty$ and $\gamma\to 0$. It is natural to expect that these limits are related to the two consecutive limits (as it is the case for nonlinear rod models as shown in \cite{neukamm2012rigorous,bauer2020derivation}). In Theorems~\ref{T:gammainfty} and~\ref{T1} we prove that in the case $\gamma\to\infty$ we indeed recover the consecutive $\Gamma$-limit $\lim_{h\to 0}\lim_{\eps\to 0}\mathcal E_{h,\eps}$. The result shows that the diagram in Figure~\ref{Fig:Diagram} is commutative in this regime. On the other hand, we also prove that the limit $\gamma\to 0$ is different from the consecutive $\Gamma$-limit $\lim_{\eps\to 0}\lim_{h\to 0}\mathcal E_{h,\eps}$, and thus the diagram is non-commutative in that regime. We note that the same observation has been made for plates without prestrain. The positive commutativity result suggests that the bending model $\mathcal E^\gamma_{\hom}$ for $\gamma\gg1$ can be approximated by the consecutive limit $\mathcal E^\infty_{\hom}$, whose effective properties are easier to compute and analyze. In view of this, it is relevant to investigate the rate of convergence as $\gamma\to\infty$. In Theorem~\ref{T:quant} and~\ref{T:quant2}  we prove a convergence rate of the effective properties as $\gamma\to\infty$. The rate of convergence is confirmed by a simulation study that we provide in Subsection~\ref{Ss:numerical}. In the following we give a more detailed presentation of our results. 

\subsection{The three-dimensional model} 

The starting point of our analysis is the following energy functional of nonlinear elasticity:
\begin{equation}
	\mathcal{E}_{h, \eps}(\deform) := \frac{1}{h^{2}} \int_{\Omega} W(x',x_3, \tfrac{x'}{\eps}, \nabla_{h}\deform(x)(\mathrm{Id}+hB(x',x_3, \tfrac{x'}{\eps}))^{-1}) \,\dd x.
\end{equation}
Here, $\Omega := S \times \left( -\frac{1}{2}, \frac{1}{2} \right)$ describes the reference domain of the plate rescaled to unit thickness and $S\subseteq\R^2$ is a Lipschitz domain; points in $\Omega$ are denoted by $x=(x', x_{3})$ with $x'\in S$; $W$ denotes a frame-invariant stored energy function with a single, non-degenerate energy-well at $\SO 3$ (for the exact hypothesis see Definition \ref{Def:StoredEnergyFct} below). The parameter $h$ denotes the thickness of the plate. The scaled gradient 
$\nabla_{h} \deform := \left( \partial_{x_{1}} \deform, \partial_{x_{2}} \deform, \frac{1}{h} \partial_{x_{3}} \deform \right)$ is due to the rescaling to unit thickness.

The scaling factor $h^{-2}$ corresponds to the bending regime: The elastic energy (normalized by volume) of a bending deformations of a plate of thickness $h$ scales with $h^2$. Finally, $B$ denotes a tensor field that describes an incremental prestrain that might be present in the reference configuration. We note that $B$ is scaled with $h$ leading to a prestrain that is compatible with the bending regime. Both the stored energy function $W(x,y,F)$ and the prestrain $B(x,y)$ are assumed to be  periodic in $y\in\R^2$. The functional $\mathcal E_{h,\eps}$ thus describes a composite that is locally periodic in in-plane directions.

\subsection{(Non-) commutativity}
The various $\Gamma$-limits of $\mathcal E_{h,\eps}$ are summarized in the diagram of Figure~\ref{Fig:Diagram}.
\begin{figure}
\centering
\begin{tikzpicture}[scale=1.3]
	\node (tl) at (0,0) {$\mathcal{E}_{h,\e}$};
	\node (tr) at (4,0) {$\mathcal{E}_{h,\hom}$};
	\node (bl) at (0,-4) {$\mathcal{E}_{\eff,\e}$};
	\node (brgamma) at (2.6,-2.6) {$\mathcal{E}^\gamma_{\hom}$};
	\node (brinf) at (4,-2.5) {$\mathcal{E}^\infty_{\hom}$};
	\node (br0) at (2,-4) {$\mathcal{E}_{\rm NO}\cdots \mathcal{E}^0_{\hom}$};
	\node (5) at (1.4,-1) {\footnotesize\rotatebox{-45}{$h,\e \to 0, \frac{h}{\e} \to \gamma$}};
	\node[red] (6) at (3,-3.1) {\footnotesize\rotatebox{45}{$0\leftarrow\gamma\to\infty$}};
	\draw[->]  (tl) -- (tr) node[above, midway] {\footnotesize$\e \to 0$};
	\draw[->]  (bl) -- (br0) node[above, midway] {\footnotesize$\e \to 0$};
	\draw[->]  (tl) -- (bl) node[left, midway] {\footnotesize$h \to 0$};
	\draw[-> ,red]  (tr) -- (brinf) node[right, midway] {\footnotesize$h \to 0$};
	\draw[->]  (tl) -- (brgamma);
	\draw[<->, red] (br0)++(0.4,0.22) arc(180:90:1.2);
\end{tikzpicture}
\caption{Diagram indicating the possible limits as $\eps, h \to 0$. Black arrows indicate limits that are already present in the literature. Red arrows indicate limits proved in this paper.}
\label{Fig:Diagram}
\end{figure}
For the purpose of our paper the simultaneous limit
\begin{equation*}
	\mathop{\lim}\limits_{(h,\eps)\to 0\atop \frac{h}{\eps}\to \gamma}\mathcal E_{h,\eps}=\mathcal E^\gamma_{\hom}
\end{equation*}
is especially relevant: In \cite{bohnlein2023homogenized} it is shown that in the case $\gamma\in(0,\infty)$ the convergence holds and the limit is a bending energy of the form
\begin{equation}
	\label{eq:limitgamma}
	\mathcal{E}^{\gamma}_{\hom}(\deform)\colonequals
	\begin{cases}
		\displaystyle \int_{\Omega} Q^{\gamma}_{\eff} \left( x',  \II_{\deform} - B_{\eff}^{\gamma} \right)\dd x' + \mathcal I_{\rm res}^\gamma&\text{if } \deform\in W^{2,2}_{\iso}(S;\R^3),\\
		+\infty&\text{otherwise.}
	\end{cases}
\end{equation}
Here, $W^{2,2}_{\rm iso}(S;\R^3):=\{\deform\in W^{2,2}(S;\R^3)\,:\,\nabla\deform^\top\nabla\deform=\mathrm{Id}\}$ denotes the space of bending deformations, and $\II_{\deform}:=\nabla \deform^{T} \nabla b_\deform$ (with $b_\deform:= \partial_{1} \deform \wedge\partial_{1} \deform $ the unit normal) denotes the second fundamental form of the surface parametrized by $\deform$. Moreover, $Q^{\gamma}_{\eff}$ is a positive semi-definite quadratic form, $B^{\gamma}_{\eff}$ is the preferred bending, and $\mathcal I_{\rm res}^{\gamma}$ is a $\deform-$independent residual energy. These quantities are obtained by linearization, relaxation and homogenization from $W$ and $B$ as outlined in Section~\ref{sect:Qeff}. The derivation of $\mathcal E^\gamma_{\hom}$ in \cite{bohnlein2023homogenized} builds on various earlier works. In particular, the case without homogenization and without prestrain (i.e., $W(x,y,F)$ is independent of $y$ and $B\equiv 0$) is covered in the seminal paper \cite{FJM02}. An extension that includes homogenization without prestrain can be found in \cite{hornung2014derivation}, while the case without homogenization but with prestrain is discussed in \cite{padilla2022dimension,agostiniani2019heterogeneous,schmidt2007plate}.

As a first main result we prove the following continuity result w.r.t. the scaling parameter $\gamma$:
\begin{theorem}
	\label{T:gammainfty}
	For $\gamma\in[0,\infty]$ let $\mathcal E^\gamma_{\hom}$, $Q^\gamma_{\eff}$, $B^\gamma_{\eff}$, and $\mathcal I_{\rm res}^\gamma$ be defined as described in Section~\ref{sect:Qeff}. Let $\bar\gamma\in[0,\infty]$. Then as $\gamma \to \bar\gamma$ we have
	\begin{align*}
		&B^{\gamma}_{\eff} \to B^{\bar\gamma}_{\eff},\qquad Q^{\gamma}_{\eff}(G) \to Q^{\bar\gamma}_{\eff}(G)\qquad\text{for all }G\in\R^{2\times 2},\\
		&\mathcal I_{\rm res}^{\gamma} \to \mathcal I_{\rm res}^{\bar\gamma}.
	\end{align*}
	Furthermore, $\mathcal E^{\gamma}_{\hom}$ $\Gamma(L^2(S))$-converges to $\mathcal E^{\bar\gamma}_{\hom}$.
\end{theorem}
See Section~\ref{sect:qual} for the proof.

Next, we compare the functional $\mathcal E^\gamma_{\hom}$ in the extremal cases $\gamma\in\{0,\infty\}$ with the consecutive limits. We start with $\gamma=\infty$, which we compare with the consecutive limit $\lim_{h\to 0}\lim_{\eps\to 0}\mathcal E_{h,\eps}$. For this purpose, we first discuss the limit $\lim_{\eps\to 0}\mathcal E_{h,\eps}$. 

The homogenization of non-convex integral functionals was originally addressed in \cite{muller1987homogenization} and revisited in \cite{baia2007limit} using two-scale convergence and $\Gamma-$convergence, and allowing for space-dependent functionals. The homogenized limit of $\mathcal{E}_{h,\eps}$ as $\epsilon \to 0$ for fixed $h > 0$ is given by
\begin{equation}
\label{eq:homenergy}
	\mathcal{E}_{h, \hom}(\deform) := \int_{\Omega} W^h_\ho(x', x_3, \nabla_{h}\deform(x)) \,\dd x,
\end{equation}
where $W^h_\ho$ is given by the multi-cell homogenization formula 
\begin{equation}
\label{eq:formW}
	W^h_\ho(x', x_3, F) := \inf_{k \in \N} \inf_{\varphi \in W^{1,2}_\per(k {\mathcal{Y}}, \R^3)} \fint_{k {\mathcal{Y}}} W(x', x_3, y, (F +  \iota(\nabla'\varphi({y})))(\mathrm{Id}+hB(x', x_3,y))^{-1}) \,\dd {y}.
\end{equation}
Here, $y=(y_1,y_2)$ and $W^{1,2}_\per(k {\mathcal{Y}}, \R^3)$ denotes the Sobolev space of functions that are in-plane periodic, i.e., periodic w.r.t.~$y$. Furthermore, $\iota (\nabla'\varphi)$ denotes the $3 \times 3$ matrix whose left and middle columns are given by $\nabla'\varphi=(\partial_{y_1}\varphi,\partial_{y_2}\varphi)$, and whose right column is $0$. 

Finer properties of the functional $\mathcal{E}_{h, \hom}$ in the case $B=0$ were analyzed in \cite{muller2011commutability}. The main result of $\cite{muller2011commutability}$ is that homogenization and linearization commute, i.e. that the homogenization of the linearized functional is the same as the linearization of the homogenized functional. This analysis was extended to include a general prestrain $B(x', x_3,y)$ in \cite{neukamm2025linearization}. 

Building upon the work of \cite{neukamm2025linearization} and \cite{padilla2022dimension} we are able to identify the consecutive limit $\lim_{h\to 0}\lim_{\eps\to 0}\mathcal E_{h,\eps}$, which we state in the following theorem. 

\begin{theorem}[Consecutive limit]
\label{T1}
  Assume that $W$ satisfies Definition \ref{Def:StoredEnergyFct} and let $ \mathcal{E}_{h, \hom}$ be given by equation \eqref{eq:homenergy}. Then the following statements hold:
  \begin{enumerate}[(\alph*)]
  \item[(a)] \label{item:T1:compactness}
   (Compactness). Let $(\deform_h)\subset L^2(\Omega;\R^3)$ be a sequence with equibounded energy, i.e.,
    \begin{equation}\label{eq:equibounded}
      \limsup_{h\to0}  \mathcal{E}_{h, \hom}(\deform_h)<\infty.
    \end{equation}
    Then there exists $\deform\in W^{2,2}_{\iso}(S;\R^3)$ and a subsequence (not relabeled) such that
    \begin{subequations}\label{T1:conv}
      \begin{alignat}{2}
      \deform_h-\fint_{\Omega}\deform_h\dd x & \to \deform & \qquad & \text{in $L^2(\Omega)$},\\
      \nabla_h \deform_h & \to (\nabla'\deform,b_\deform) && \text{in $L^2(\Omega)$}.
    \end{alignat}
    \end{subequations}
    Here and below, $b_\deform \colonequals\partial_1 \deform \wedge \partial_2 \deform$ denotes the unit normal.
  \item[(b)] \label{item:T1:lower_bound} (Lower bound). If $(\deform_h)\subset L^2(\Omega;\R^3)$ is a sequence with $\deform_h-\fint_{\Omega}\deform_h\dd x\to \deform $ in $L^2(\Omega)$, then
    \begin{equation*}
      \liminf_{h\to0}  \mathcal{E}_{h, \hom}(\deform_h)\geq {\mathcal E}^{\infty}_{\hom} (\deform).
    \end{equation*}
  \item[(c)] \label{item:T1:recovery_sequence} (Recovery sequence). For any $\deform\in W^{2,2}_{\iso}(S;\R^3)$ there exists a sequence $(\deform_h)\subset W^{1,\infty}(\Omega;\R^3)$ with $\deform_h\to \deform$ strongly in $W^{1,2}(\Omega;\R^3)$ such that
    \begin{equation}\label{T1:c}
      \lim_{h\to0}  \mathcal{E}_{h, \hom}(\deform_h)={\mathcal E}^{\infty}_{\hom} (\deform).
    \end{equation}
  \end{enumerate}
\end{theorem}%
\noindent
The proof is found in Section~\ref{sect:proofT1}. 

Together, Theorems \ref{T:gammainfty} and \ref{T1} imply that the limits $\lim_{h\to 0}\lim_{\eps\to 0}\mathcal E_{h,\eps}$ and $\lim_{\gamma \to \infty} \mathop{\lim}\limits_{(h,\eps)\to 0\atop \frac{h}{\eps}\to \gamma}\mathcal E_{h,\eps}$ are the same. Hence, there is commutativity at $\gamma = \infty$. 

On the other hand, the picture for $\gamma = 0$ is highly non-commutative. Firstly, we may look at the consecutive limit $\lim_{\eps\to 0} \lim_{h\to 0} \mathcal E_{h,\eps}$. In the absence of prestrain (i.e. $B=0$), the limit $\lim_{h\to 0} \mathcal E_{h,\eps}$ is simply a nonlinear bending energy functional with oscillating coefficients (see \cite{FJM02}), in other words an energy given by 
\begin{equation}
\mathcal{E}_{\eff, \eps} (\deform)=  \begin{cases}
    \displaystyle \int_{S} Q_{\eff} \left( \frac{x'}{\eps},  \II_{\deform} \right)\dd x'&\text{if } \deform\in W^{2,2}_{\iso}(S;\R^3),\\
    +\infty&\text{otherwise.}
  \end{cases}
\end{equation}
 The limit as $\eps \to 0$ of this sequence of energies is given by a homogenized nonlinear bending theory, derived in \cite{neukamm2015homogenization}. 
 For cylindrical deformations it is of the form
\begin{equation}
\label{eq:generalbending}
     \mathcal{E}_{\rm NO}(\deform ) \colonequals  \begin{cases}
        \displaystyle \int_{S} V \left( \II_{\deform} \right)\dd x'&\text{if } \deform\in W^{2,2}_{\iso}(S;\R^3),\\
        +\infty&\text{otherwise,}
  \end{cases}
\end{equation}\\
where $V$ is non-quadratic (see \cite{neukamm2015homogenization} equations (5) and (6)).
Hence, the limits $\lim_{\eps\to 0}\lim_{h\to 0}\mathcal E_{h,\eps}$ and $\lim_{\gamma\to 0} \mathop{\lim} \limits_{(h,\eps)\to 0\atop \frac{h}{\eps}\to \gamma}\mathcal E_{h,\eps}$ are not the same, and there is no commutativity at $\gamma = 0$. Furthermore, we may ask about the limits of $\mathcal{E}_{h, \eps}$ in the regime $h,\eps \to 0$ simultaneously, with $\frac{h}{\epsilon} \to 0$, which also corresponds formally to $\gamma = 0$. The regime $h \ll \epsilon^{2}$ was derived in \cite{cherdantsev2015bending}, while the regime $\epsilon^{2} \ll h \ll \epsilon$ was derived in \cite{velvcic2015derivation}. The derivation of a nonlinear bending theory in the regime $h \simeq \epsilon^{2}$ is still an open problem. The limits derived in \cite{cherdantsev2015bending} and \cite{velvcic2015derivation} are different among themselves, and different from the other two limits discussed before which formally correspond to $\gamma = 0$. 

\subsection{Rate of convergence for $\gamma \to \infty$}
\label{sect:rate}

We investigate the rate of convergence of $\mathcal{E}^{\gamma}_{\hom}$ as $\gamma \to \infty$. We do not expect that a nontrivial rate holds for general elastic laws and general prestrains, but we will show that, assuming additional material symmetry and regularity in $x_3$, it is possible to show a polynomial rate of convergence to the limit. 

For the precise statement of the result we need to introduce some notation. The coefficients $Q_{\rm eff}^\gamma$ and $B^\gamma_{\hom}$ of the effective bending model depend on the quadratic term $Q$ in the expansion at identity of the stored energy function $W$, and the prescribed prestrain tensor $B$. Both, $Q$ and $B$ depend on the macroscopic in-plane variable $x'\in S$, as well as on the (rescaled) out-of plane variable $x_3\in I:=\left( -\frac{1}{2}, \frac{1}{2} \right)$, and the microscopic in plane variables $y=(y_1,y_2)\in\mathcal Y$. Here, $\mathcal{Y}$ denotes the $2d$ torus $\R^2/\Z^2$, which we identify (as a set) with $(0,1)^2$. We therefore introduce the notation $\Box \colonequals I \times \mathcal{Y}$. We consider the following class of quadratic forms and prestrains:

\begin{definition}
	\label{def:class}
	The class $\mathcal{Q}(\alpha,\beta)$ consists of all quadratic forms $Q$ on $\R^{3\times 3}$ such that
	\begin{equation*}
		\alpha|\sym G|^2\leq Q(G)\leq\beta|\sym G|^2
		\qquad
		\text{for all }G\in\R^{3\times 3},
	\end{equation*}
	where $\sym{G} \colonequals \frac{1}{2}(G + G^\top)$ is the symmetric part of $G$. We call an elastic law $Q$ \emph{orthotropic} if 
	\begin{equation*}
			Q(\iota(F) + d \otimes e_{3}) =Q(\iota(F)) + Q(d \otimes e_{3})
	\end{equation*}
	for all $F\in \R^{2 \times 2}$ and $d \in \R^{3}$.  We associate with each $Q \in \mathcal{Q}(\alpha,\beta)$ the fourth-order stiffness tensor $\mathbb L\in\operatorname{Lin}(\R^{3\times 3},\R^{3\times 3})$ defined by the polarization identity
	\begin{equation}
		\mathbb L H:G\colonequals\frac12\big(Q(H+G)-Q(H)-Q(G)\big),
	\end{equation}
	where $:$ denotes the standard scalar product in $\R^{3\times 3}$.
\end{definition} 

\begin{definition}[Admissible quadratic form and prestrain]\label{D:admissible}
	\begin{itemize}
		\item[(a)]
		A Borel function $Q:S \times I \times \mathcal{Y} \times\R^{3\times 3}\to\R$ is called an
		\emph{admissible quadratic form}, if $Q(x', x_3,y,\cdot)$ is a quadratic form
		of class $\mathcal{Q}(\alpha,\beta)$ for a.e.\ $(x', x_3,y) \in S \times I \times \mathcal{Y}$.
		
		\item[(b)] A Borel function $B:S \times I \times \mathcal{Y} \to\R^{3\times 3}_{\sym}$ is called an \emph{admissible prestrain}, if $B$ is square integrable. 
	\end{itemize}
\end{definition}

We now state the main result regarding the rates of convergence. 
\begin{theorem}
\label{T:quant}
Let $Q$ and $B$ be an admissible quadratic form and prestrain, respectively. Let $Q^\gamma_{\eff}$, $B^\gamma_{\rm eff}$ and $\mathcal I^\gamma_{\rm eff}$ be defined as in Section~\ref{sect:Qeff}. 
Assume that $Q$ is orthotropic and that for some $\alpha \in (0,1]$
\begin{equation}
\label{eq:alpharegularity}
\sup_{x' \in S, y \in \mathcal{Y}; x_{3}^{1}, x_{3}^{2} \in I } \frac{\left| \mathbb L(x', x_{3}^{1}, y) - \mathbb L(x', x_{3}^{2}, y)\right|}{|x_{3}^{2} - x_{3}^{1}|^{\alpha}} < \infty.
\end{equation}
Then there holds:
\begin{itemize}
\item[(a)] There exists a constant $C$ such that for all $\gamma \in (0,\infty )$,
\begin{equation}
\label{eq:q-rate-convergence}
Q^{\infty}_{\eff} (x',G)\leq Q^{\gamma}_{\eff} (x',G) \leq Q^{\infty}_{\eff} (x',G) + \frac{C |G|^{2}}{\gamma^{2 \alpha}},
\end{equation} 
for a.e. $x' \in S$ and all $G\in\R^{2\times 2}$.
\item[(b)] There exists a constant $C$ such that such that for all $\gamma \in (0,\infty )$ and for a.e. $x' \in S$,
\begin{equation}
\left| B^{\gamma}_{\eff}(x') - B^{\infty}_{\eff}(x') \right| \leq \frac{C}{\gamma^{\alpha}}. 
\end{equation}
\item[(c)] Assume that 
\begin{equation}
	B\in\begin{cases}
		L^{\infty}(S, C^{\alpha} (I, L^{2}(\mathcal{Y}, \R^{3})))&\text{if }\alpha\in(0,1),\\
		L^{\infty}(S, W^{1, \infty} (I, L^{2}(\mathcal{Y}, \R^{3})))&\text{if }\alpha=1.
	\end{cases}
\end{equation}
Then there exists a constant $C$ such that for all $\gamma \in (0,\infty )$,
\begin{equation}
\left| \mathcal I_{\rm res}^{\gamma} - \mathcal I_{\rm res}^{\infty} \right| \leq \frac{C}{\gamma^{\alpha}}. 
\end{equation}
\end{itemize}
\end{theorem}
\noindent
\begin{remark}
\label{rem:even}
    In the setting of Theorem \ref{T:quant}, if $Q$ is not orthotropic, but is even in $x_{3}$, it still holds that there exists a constant $C$ such that for all $\gamma \in (0,\infty )$,
\begin{equation}
 Q^{\gamma}_{\eff} (x',G) \leq Q^{\infty}_{\eff} (x',G) + \frac{C |G|^{2}}{\gamma^{2 \alpha}},
\end{equation} 
for a.e. $x' \in S$ and all $G\in\R^{2\times 2}$.
\end{remark}
\noindent
The proof of Theorem \ref{T:quant} is found in Section~\ref{sect:quant}, and the proof of Remark \ref{rem:even} is contained in the proof of Theorem \ref{T:quant}, Substep 2.1. 
\smallskip

If we do not assume any additional structure on $Q$, i.e. if we drop the hypothesis that $Q$ is orthotropic or even in $x_{3}$, our methods only allow us to obtain a partial result for the convergence of $Q^{\gamma}_{\eff}$, see Remark \ref{T:quant3}. However, the assumption that $Q$ satisfies additional symmetry or structure (i.e. that $Q$ is orthotropic) is very natural, and has been used, for example, in \cite{bohnlein2023homogenized}. 

\begin{remark}
\label{T:quant3}
     Assume the same hypotheses as in Theorem \ref{T:quant} except that $Q$ is orthotropic and $\alpha > \frac{1}{2}$, then there exists a constant $C$ such that for all $\gamma \in (0,\infty )$,
    \begin{equation}
    Q^{\gamma}_{\eff} (x',G) \leq Q^{\infty}_{\eff} (x',G) + \frac{C |G|^{2}}{\gamma^{ \alpha (2 \alpha - 1)}},
    \end{equation} 
    for a.e. $x' \in S$ and all $G\in\R^{2\times 2}$.
\end{remark}
\noindent
The proof is found in Section~\ref{sect:quant}. 
\smallskip

Next to H\"older-continuity in $x_3$, we consider material laws that are piece-wise constant in $x_3$ direction. These are especially interesting for applications, since they describe laminates. Although these laminates are not covered by Theorem \ref{T:quant}, it is possible to use the same ideas to prove a rate of convergence for laminates.

\begin{theorem}
\label{T:quant2}
Let $Q$ and $B$ be an admissible quadratic form and prestrain, respectively. Let $Q^\gamma_{\eff}$, $B^\gamma_{\rm eff}$ and $\mathcal I^\gamma_{\rm eff}$ be defined as in Section~\ref{sect:Qeff}. 
Assume that $Q$ is orthotropic and piece-wise constant in $x_{3}$, that is:
\begin{equation}
\label{eq:Qpiecewise}
Q (x',x_{3}, y, G) = \sum_{i=1}^{N} \mathbf{1}_{[p_{i}, p_{i+1})}(x_{3}) Q_{i}(x',y, G),
\end{equation}
where $\{p_{i}\}_{i=1}^{N}$ are the endpoints of a partition of $I$. 

Then there holds:
\begin{itemize}
	\item[(a)] There exists a constant $C$ such that for all $\gamma \in (0,\infty )$,
	\begin{equation}
		Q^{\infty}_{\eff} (x',G)\leq Q^{\gamma}_{\eff} (x',G) \leq Q^{\infty}_{\eff} (x',G) + \frac{C |G|^{2}}{\gamma},
	\end{equation} 
	for a.e. $x' \in S$ and all $G\in\R^{2\times 2}$.
	\item[(b)] There exists a constant $C$ such that such that for all $\gamma \in (0,\infty )$ and for a.e. $x' \in S$,
	\begin{equation}
		\left| B^{\gamma}_{\eff}(x') - B^{\infty}_{\eff}(x') \right| \leq \frac{C}{\gamma^{\frac12}}. 
	\end{equation}
	\item[(c)] Assume that 
	$B$ is given by
	\begin{equation}
		\label{eq:Bpiecewise}
		B(x',x_{3}, y) = \sum_{i=1}^{N} \mathbf{1}_{[p_{i}, p_{i+1})}(x_{3}) B_{i}(x',y),
	\end{equation}
	with $B_{i} \in L^{2}(S \times \mathcal{Y}, \R^{3})$.
	Then there exists a constant $C$ such that for all $\gamma \in (0,\infty )$,
	\begin{equation*}
		\left| \mathcal I_{\rm res}^{\gamma} - \mathcal I_{\rm res}^{\infty} \right| \leq \frac{C}{\gamma^{\frac12}}. 
	\end{equation*}
\end{itemize}
\end{theorem}
\begin{remark}
\label{rem:even2}
In the setting of Theorem \ref{T:quant2}, if $Q$ is not orthotropic, but is even in $x_{3}$, it still holds that there exists a constant $C$ such that for all $\gamma \in (0,\infty )$,
\begin{equation}
 Q^{\gamma}_{\eff} (x',G) \leq Q^{\infty}_{\eff} (x',G) + \frac{C |G|^{2}}{\gamma},
\end{equation} 
for a.e. $x' \in S$ and all $G\in\R^{2\times 2}$.
\end{remark}
We will not give the full proof of Theorem \ref{T:quant2} since it is very similar to the proof of Theorem \ref{T:quant}. Instead, we will only indicate the points at which the proofs differ. This is done in Section~\ref{sect:quant}.  

\subsection{Examples}
\label{Ss:numerical}

In this subsection, we look at two examples of the theory developed in Subsection \ref{sect:rate}. The first example satisfies the hypotheses of Theorems \ref{T:quant}, and we will prove analytically that the rates of convergence proved are optimal. For the second example, we will consider a setup that does not satisfy the hypotheses of Theorems \ref{T:quant} or \ref{T:quant2}, and derive rates of convergence numerically. 

\subsubsection{Analytic example} \label{sec:analytic-example}

The first example we will analyze in this section is based on the setup considered in \cite{bohnlein2023homogenized}, Section 4. In this example, the elastic law is as in \cite{bohnlein2023homogenized}: it is isotropic, homogeneous and given by 
\begin{equation}
\label{eq:Q1}
    Q\left( x', x_{3}, y, G \right) = 2 \mu (y_{1}) \left| G \right|^{2},
\end{equation}
where the function $\mu :  \mathbb{T} \to \mathbb{R}$ is given by 
\begin{equation}
\label{eq:Q2}
    \mu (y_{1}) = 
    \begin{cases}
        \mu_{1} &\text{ if } |y_{1}|< \frac{\theta}{2} \\
        \mu_{2} &\text{ otherwise.} 
    \end{cases}
\end{equation}
Figure \ref{fig:examplesetup} shows an illustration of this setup. Note that this setup is homogeneous (there is no $x'$ dependence) and so the effective quantities are also homogeneous. 

\begin{figure}[h]
    \centering
    \includegraphics[width=.5\textwidth, trim={20cm 0 0 0 }, clip]{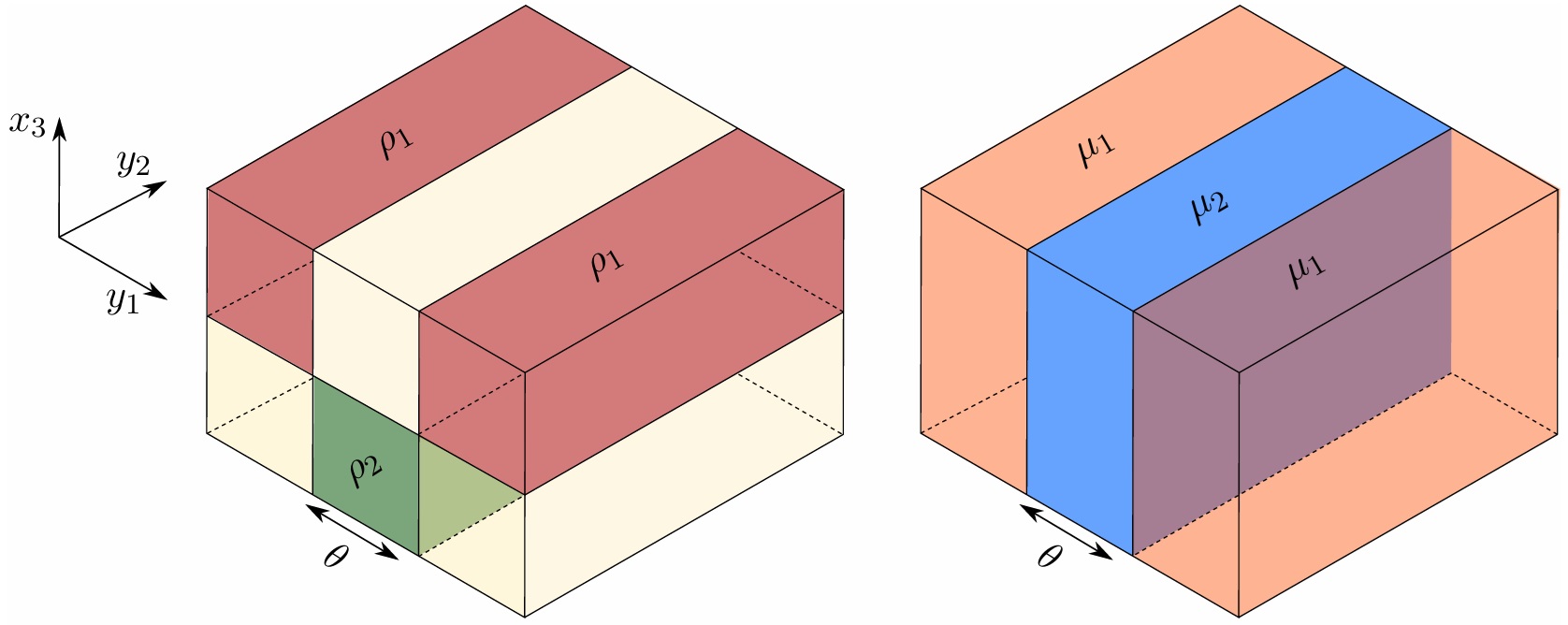}
    \includegraphics[width=.4\textwidth, trim={17cm 0cm 17cm 3cm }, clip]{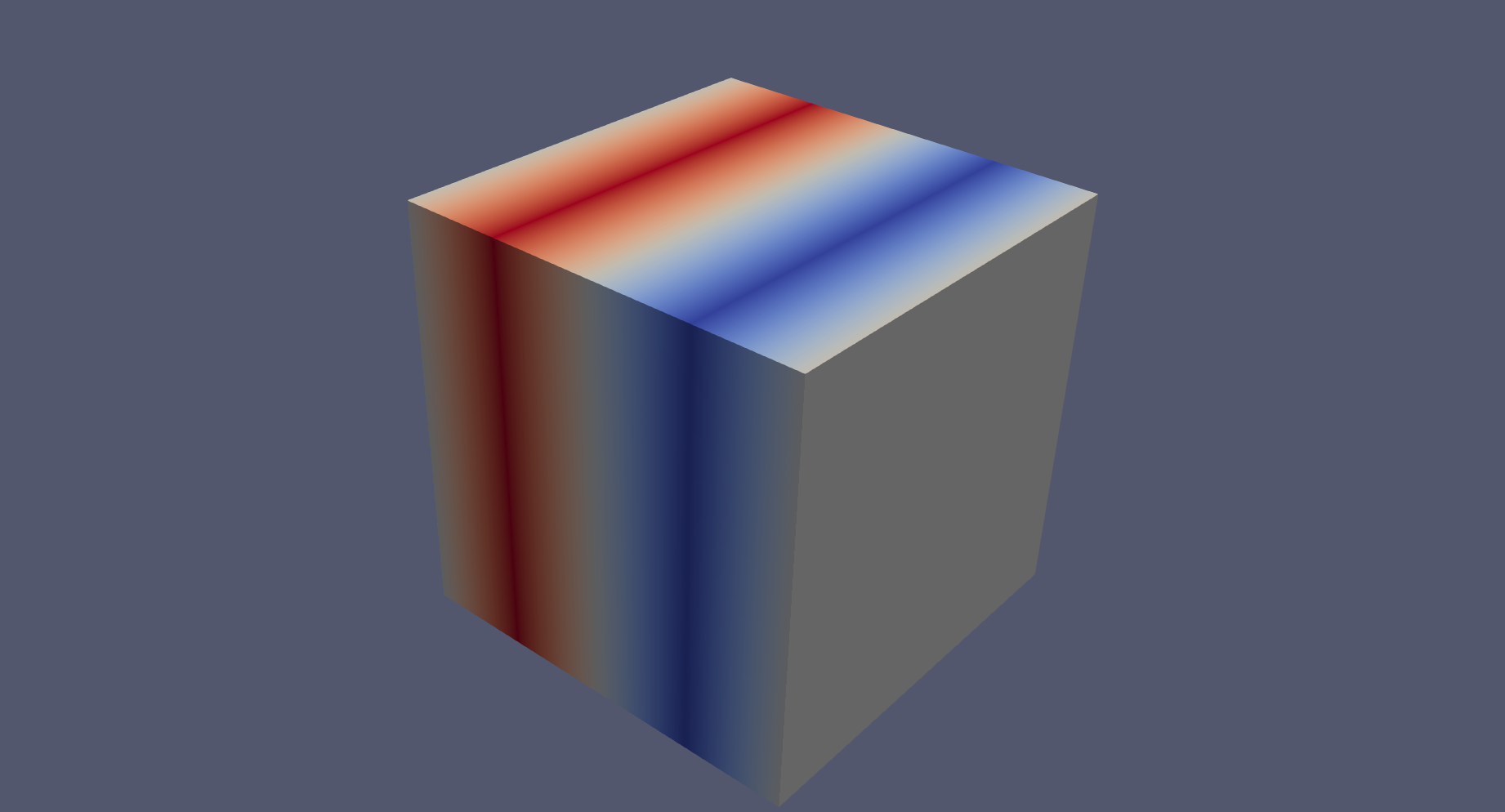}
    \caption{
    Illustration of the example. Left: Elastic law $\mu$, Right: Plot of the $(2,3)$ component of the prestrain $B_{23}(x', x_{3}, y)$ with parameters $\theta = 1/2, \mu_1 = 1, \mu_2 = 2$. The $2,3$ component is the only nonzero component of $B$ in this example. The color here ranges from blue for -0.12 to red for 0.12.
    }
    \label{fig:examplesetup}
\end{figure}

The convergence rate for $B^{\gamma}_{\eff}$ predicted by Theorem \ref{T:quant} as $\gamma \to \infty$ can clearly not be optimal for \emph{all prestrains} $B$. Indeed, if $B=0$, then $B^{\gamma}_{\eff}$ is $0$, and in particular the convergence rate is faster than the one in Theorem \ref{T:quant}. In \cite{bohnlein2023homogenized}, the authors consider a piece-wise constant prestrain that also results in a $B^{\gamma}_{\eff}$ which is independent of $\gamma$. Given the previous observation, we ask if the convergence rate for $B^{\gamma}_{\eff}$ in Theorem \ref{T:quant} is optimal for \emph{one prestrain} $B$. The answer turns out to be positive, and the indicated prestrain is given by 
\begin{equation}
\label{eq:B1}
    B(x', x_{3}, y) = (0,0,\partial_3\varphi^3_\infty)
\end{equation}
where $\varphi^3_\infty$ is the corrector associated with $G_3$ at $\gamma = \infty$ in the sense of Lemma~\ref{L:corrector}. We illustrate the third column of this prestrain in Figure \ref{fig:examplesetup}.

Note that in this example, $Q$ is independent of $x_{3}$, and therefore satisfies equation \eqref{eq:alpharegularity} with $\alpha=1$. Hence, Theorem \ref{T:quant} gives a rate of convergence of $\frac{1}{\gamma^{2}}$ for $Q^{\gamma}_{\eff} $ and $\frac{1}{\gamma}$ for $B^{\gamma}_{\eff}$. The proposition below shows that these rates are sharp for this example.
\begin{proposition}
\label{prop:example}
    Let $Q$ be given by equations \eqref{eq:Q1} and \eqref{eq:Q2} and $B$ given by equation \eqref{eq:B1}. Then 
    \begin{itemize}
        \item[(a)] There exist constants $c,C$ such that
        \begin{equation}
            \frac{c \left| G \right|^{2}}{\gamma^{2}} \leq \left| Q^{\gamma}_{\eff} (G) - Q^{\infty}_{\eff} (G) \right| \leq \frac{C \left| G \right|^{2}}{\gamma^{2}}. 
        \end{equation}

        \item[(b)] There exist constants $c,C$ such that
        \begin{equation}
             \frac{c}{{\gamma}} \leq \left| B^{\gamma}_{\eff} - B^{\infty}_{\eff} \right| \leq \frac{C}{{\gamma}}.
        \end{equation}
    \end{itemize}
\end{proposition}
The proof is found at the end of Section~\ref{sect:quant}. 



We now present computational verification of the rates of convergence for this example. The solution to the corrector problem for $\gamma <\infty$ (see \cite[Eq.\ (20)]{bohnlein2023homogenized}) is computed with a finite element method using the C++ library Dune \cite{blatt2016distributed} for values $\gamma = 2^k$ for $k=0,\ldots, 18$. We later show in the proof of Proposition \ref{prop:example} that the corrector $\varphi^3_\gamma$ is independent of $\gamma$, so the prestrain $B$ in Figure \ref{fig:examplesetup} and \eqref{eq:B1} is computed from the corrector corresponding to $\gamma=2^{18}$.  For other values of $\gamma$, the effective quantities $B_{\text{eff}}^\gamma$, $Q_{\text{eff}}^\gamma$ are computed using the formulas \cite[Proposition 2.25]{bohnlein2023homogenized}. The specific parameters used were
$$
\theta = \frac{1}{2}, \quad \mu_1 = 1, \quad \mu_2 = 2.
$$
For this example, we have exact solutions for $Q^\infty_\text{eff}$ from \cite{bohnlein2023homogenized} where $q_{11} = q_{33} = 2/9$, $q_{22} = 1/4$, and $q_{ij} = 0$ for $i\neq j$. For the purposes of this simulation, $B^\infty_\eff$, was approximated by $B^\gamma_\eff$, for $\gamma = 2^{18}$. The resulting numerical value is $B^{2^{18}}_{\eff} = 1.192\times 10^{-7} G_3$.
\begin{figure}[h]
    \centering
    \includegraphics[scale=0.4]{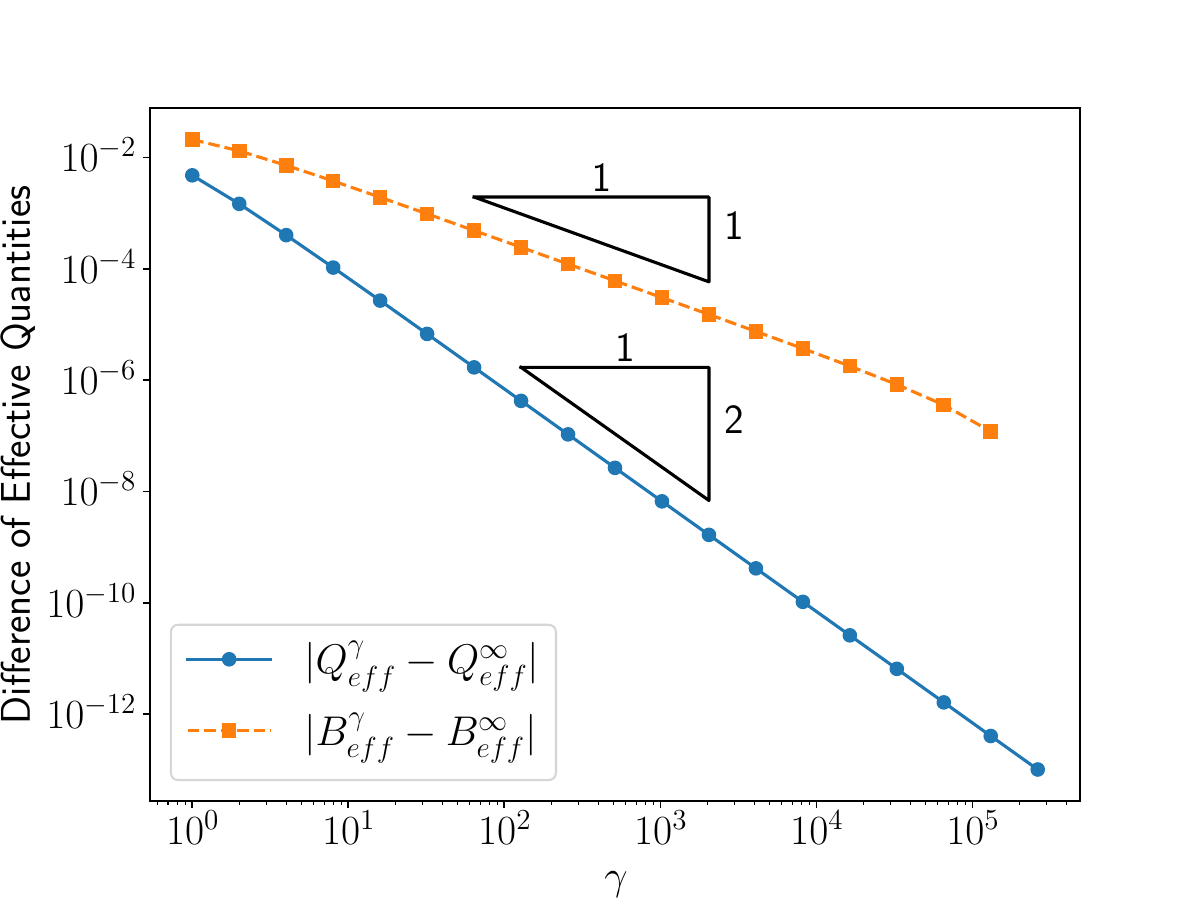}
    \caption{A log-log plot of the convergence rate for $Q^{\gamma}_{\eff},B^{\gamma}_{\eff}$. {From the slope of the line above, as indicated by the triangle, we see that the rates of convergence are approximately $\gamma^{-2}$ and $\gamma^{-1}$ respectively, which verifies that the rate of convergence in equation \eqref{eq:q-rate-convergence} of Theorem \ref{T:quant} is optimal in this example.}}
    \label{fig:convergenceQ}
\end{figure}
Figure \ref{fig:convergenceQ} shows the convergence rate of $Q^{\gamma}_{\eff}$ towards $Q^{\infty}_{\eff}$ and $B^{\gamma}_{\eff}$ towards $B^{\infty}_{\eff}$ in a log-log plot. The norm of $| Q^{\gamma}_{\eff} - Q^{\infty}_{\eff}|$ is measured as the difference between the coefficients of the corresponding matrix representations. The slope of the best-fit linear approximation in the log-log plot is approximately $-2$ for $| Q^{\gamma}_{\eff} - Q^{\infty}_{\eff}|$ and approximately $-1$ for $| B^{\gamma}_{\eff} - B^{\infty}_{\eff}|$, which verifies that the rate of convergence proved in Theorem \ref{T:quant} is optimal in this example.

\subsubsection{Numerical example}
{
Inspired by real-world examples of wood composites in \cite{bohnlein2025dimension}, we now consider an example that does not satisfy the hypotheses of Theorem \ref{T:quant2}. In particular, this setup will have a quadratic form that is not orthotropic in general. The focus of \cite{bohnlein2025dimension} was on thin sheets of wood composite bilayers and not on microheterogeneous materials, but this section explores the interaction of the length scale of the microstructure, $\varepsilon$, with the sheet thickness $h$.
}

{We consider the following quadratic form with stiffness tensor $\mathbb{L}$ \footnote{In \cite{bohnlein2025dimension}, the stiffness tensor is denoted by $\mathbb{C}$} as shown in Figure \ref{fig:wood-example} (left):
\begin{equation}\label{eq:numerical-quadratic-form2}
Q\left(x^{\prime}, x_{3}, y, G\right)= \mathbb{L}(y_1,y_2, x_3)G : G.
\end{equation}
In order to define the stiffness tensor using matrix notation, we use Voigt notation, which is a way of notating matrices and tensors. In Voigt notation, a symmetric $3\times3$ matrix $G$ is written as a vector of length 6 given by
$$\mathbf{v}(G) = (G_{11},G_{22},G_{33}, G_{23}, G_{13}, G_{12})^T.$$ 
Similarly, in Voigt notation a tensor $\mathbb{A} \in\operatorname{Lin}(\R^{3\times 3}_{sym},\R^{3\times 3}_{sym})$, is written as a $6\times 6$ matrix $\mathbf{V}(\mathbb{A})$ defined by  
$$
\mathbf{V}(\mathbb{A}) \mathbf{v}(G) : =  \mathbf{v}(\mathbb{A} G).
$$
}

{We now define the compliance tensor for European Beech $\mathbb{S}_{beech} := \mathbb{L}_{beech}^{-1}$. Wood has 3 orthogonal axes of symmetry, which are the radial direction, $e_R$, the tangential direction to the tree rings $e_T$, and the longitudinal direction, $e_L$. Choosing a coordinate frame in which the canonical axes correspond to the axes of symmetry, i.e.\ $e_1 = e_R$, $e_2 = e_T$, and $e_3=e_L$, the compliance tensor for European beech in Voigt notation is given by
\begin{equation}\label{eq:compliance-voigt}
\mathbf{V}(\mathbb{S}_{beech}(\omega))=\left(\begin{array}{cccccc}E_{1}^{-1} & -\nu_{21} E_{2}^{-1} & -\nu_{31} E_{3}^{-1} & 0 & 0 & 0 \\ -\nu_{12} E_{1}^{-1} & E_{2}^{-1} & -\nu_{32} E_{3}^{-1} & 0 & 0 & 0 \\ -\nu_{13} E_{1}^{-1} & -\nu_{23} E_{2}^{-1} & E_{3}^{-1} & 0 & 0 & 0 \\ 0 & 0 & 0 & G_{23}^{-1} & 0 & 0 \\ 0 & 0 & 0 & 0 & G_{13}^{-1} & 0 \\ 0 & 0 & 0 & 0 & 0 & G_{12}^{-1} \end{array}\right),
\end{equation}
where $E_{i}$ denotes the Young's moduli, $G_{ij}$ denotes shear moduli, and $\nu_{ij}$ denote the Poisson ratios. In equation \eqref{eq:compliance-voigt}, the elastic moduli $E_{i}, G_{ij}, \nu_{ij}$ depend on a dimensionless moisture parameter $\omega$, that can take values between $0\%$ and $100\%$. The specific dependence of the moduli used in this example are taken from \cite[Table 1]{bohnlein2025dimension} and \cite[Table B.1]{hassani2015rheological}. In this example, we set $\omega = 14.70179844\%$. Inserting these values into equation \eqref{eq:compliance-voigt} yields the following compliance tensor
\begin{equation}\label{eq:compliance-beech}
\begin{aligned}
\mathbf{V}(\mathbb{S}_{beech}&(\omega)) = \\
& 10^{-3}\cdot\begin{pmatrix}
5.92 \cdot 10^{-1} & -5.14 \cdot 10^{-1} & -1.96 \cdot 10^{-2} & 0 & 0 & 0 \\
-5.14 \cdot 10^{-1} & 1.85  & -1.57 \cdot 10^{-2} & 0 & 0 & 0 \\
-1.96 \cdot 10^{-2} & -1.57 \cdot 10^{-2} & 7.70 \cdot 10^{-2} & 0 & 0 & 0 \\
0 & 0 & 0 & 1.19  & 0 & 0 \\
0 & 0 & 0 & 0 & 7.95 \cdot 10^{-1} & 0 \\
0 & 0 & 0 & 0 & 0 & 2.25
\end{pmatrix}
\end{aligned}.
\end{equation}
}

{
In this example, the axes of symmetry will also change spatially. For more general orthogonal axes of symmetry $e_R,e_T,e_L$, the compliance tensor can be defined by using a rotation matrix $U$ and its action on a matrix $G$:
\begin{equation}\label{eq:rotated-voigt}
    \mathbb{S}_{beech}(\omega; e_R, e_T, e_L) G := U [\mathbb{S}_{beech}(\omega) (U^T GU)] U^T, \quad U := e_R\otimes e_1 +e_T\otimes e_2 +e_L\otimes e_3,
\end{equation}
where $\mathbb{S}_{beech}(\omega)$ is defined in equation \eqref{eq:compliance-voigt}. We now define  $\mathbb{L}^{-1}(y_1,y_2, x_3)$ as follows
\begin{equation}\label{eq:compliance-def}
\mathbb{L}^{-1}(y_1,y_2, x_3) = \begin{cases} \mathbb{S}_{iso} , & |y_1| > 1/4 \\
\mathbb{S}_{beech}(\omega, e_2, e_3, e_1) , & |y_1| \leq 1/4, x_3 < 0\\
\mathbb{S}_{beech}(\omega, R(\pi/4)e_2, e_3, R(\pi/4)e_1) , & |y_1| \leq 1/4, x_3 > 0.
\end{cases}
\end{equation}
We set $\mathbb{S}_{iso}$ to an isotropic elastic law with Lam\'e parameters $\mu = 1,\lambda = 0$. For the inner layers $\{|y_1| \leq 1/4, x_3 < 0\}$ and ${|y_1| \leq 1/4, x_3 > 0}$, the compliance tensor $\mathbb{S}_{beech}(\omega, e_R, e_T, e_L)$ is defined in equation \eqref{eq:rotated-voigt}. On the bottom layer, we keep the axis of the wood as $e_2,e_3,e_1$. On the top layer, we rotate about the $e_3$ axis by $\pi/4$ using the rotation matrix $R(\pi/4)$. We also note that the quadratic form $Q$ defined in equation \eqref{eq:numerical-quadratic-form2} is not orthotropic in the region ${|y_1|\leq \frac{1}{4}}$ as long as $\nu_{ij} \neq 0$ for all $i,j$. Hence, the example presented here is does not satisfy the hypotheses of Theorem \ref{T:quant2}.}

{
We consider two prestrains in this example, the first one is a uniform hydrostatic pressure on the bottom layer with no prestrain on the top layer. The second prestrain is the one previously considered in Subsection \ref{sec:analytic-example}. To be precise, the two prestrains we consider are
\begin{equation*}
B_1(x',x_3,y) = \begin{cases}
\mathrm{Id}, &x_3 < 0\\
0, &x_3 \geq0
\end{cases}, \quad B_2(x',x_3,y) = (0,0,\partial_3\varphi^3_\infty)
\end{equation*}
where $\varphi^3_\infty$ is the corrector associated with $G_3$ at $\gamma = \infty$ in the sense of Lemma~\ref{L:corrector}. 
}

{For the numerics, we apply the same procedure as the previous example but instead take the range $\gamma = 2^\ell, \ell = 0,\ldots, 9$. We then use the approximations $Q_{eff}^\infty\approx  Q_{eff}^{512}, B_{eff}^\infty\approx  B_{eff}^{512}$. We do not take larger values for $\gamma$ like the previous example because it becomes difficult to numerically approximate the effective quantities if the quadratic form is not orthotropic and $\gamma$ is large. More specifically, the error in numerically computing $Q^\gamma_{eff}$ and $B^\gamma_{eff}$ is determined by the mesh size and the $L^2$ norm of the Hessian of the corrector. In the case of non-orthotropic quadratic forms, it is possible that $\partial_{3,3}^2\varphi^3_\gamma$ becomes very large as $\gamma\to\infty$, and hence the practical accuracy of the numerical method is limited for large $\gamma$. By contrast, the example in Section \ref{sec:analytic-example} can be solved analytically and be shown to satisfy $\partial_{3,3}^2\varphi^3_\gamma = 0$, which means that the numerical method can be accurate even for large $\gamma$ in that case.
}

{
We find that the numerical values for the approximations of $Q_{eff}^{\infty},B_{eff,1}^{\infty},B_{eff,2}^{\infty}$ are 
$$
Q_{eff}^{\infty} = 
\begin{pmatrix}0.22 & 0.0038 & 1.45\\
0.0038 & 44.9 & 0.012 \\
1.45 & 0.012 & 0.22
\end{pmatrix},
 \quad B_{eff,1}^{\infty} = -1.5 G_1 - 1.47 G_2,
 \quad B_{eff,2}^{\infty} = .012 G_3. 
$$
We see in Figure \ref{fig:wood-example} (right) that as $\gamma\to\infty$, $Q_{eff}^{\gamma}\to Q_{eff}^{\infty}$ and $B_{eff,1}^{\gamma}\to B_{eff,1}^{\infty}$ with rate $\gamma^{-2}$. We also see that $B_{eff,2}^{\gamma}\to B_{eff,2}^{\infty}$ with rate $\gamma^{-1}$. This is numerical evidence to suggest that Theorem \ref{T:quant2} may hold for a broader class of quadratic forms that are not orthotropic.
}
\begin{figure}[h]
\begin{minipage}{.45\textwidth}
\includegraphics[width=1\textwidth]{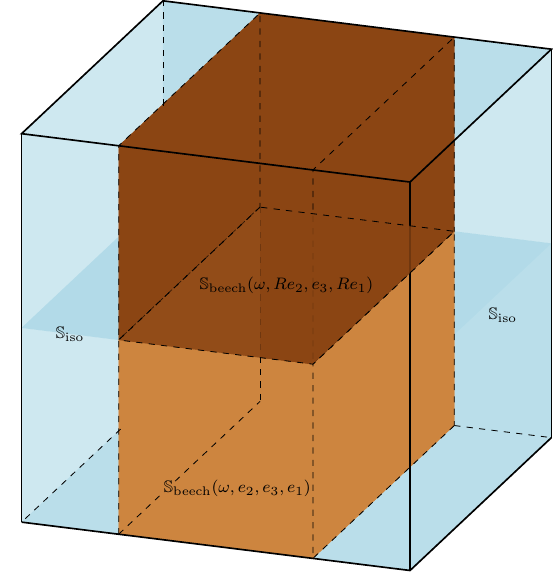}
\end{minipage}
\hfill
\begin{minipage}{.5\textwidth}
\includegraphics[width=1\textwidth]{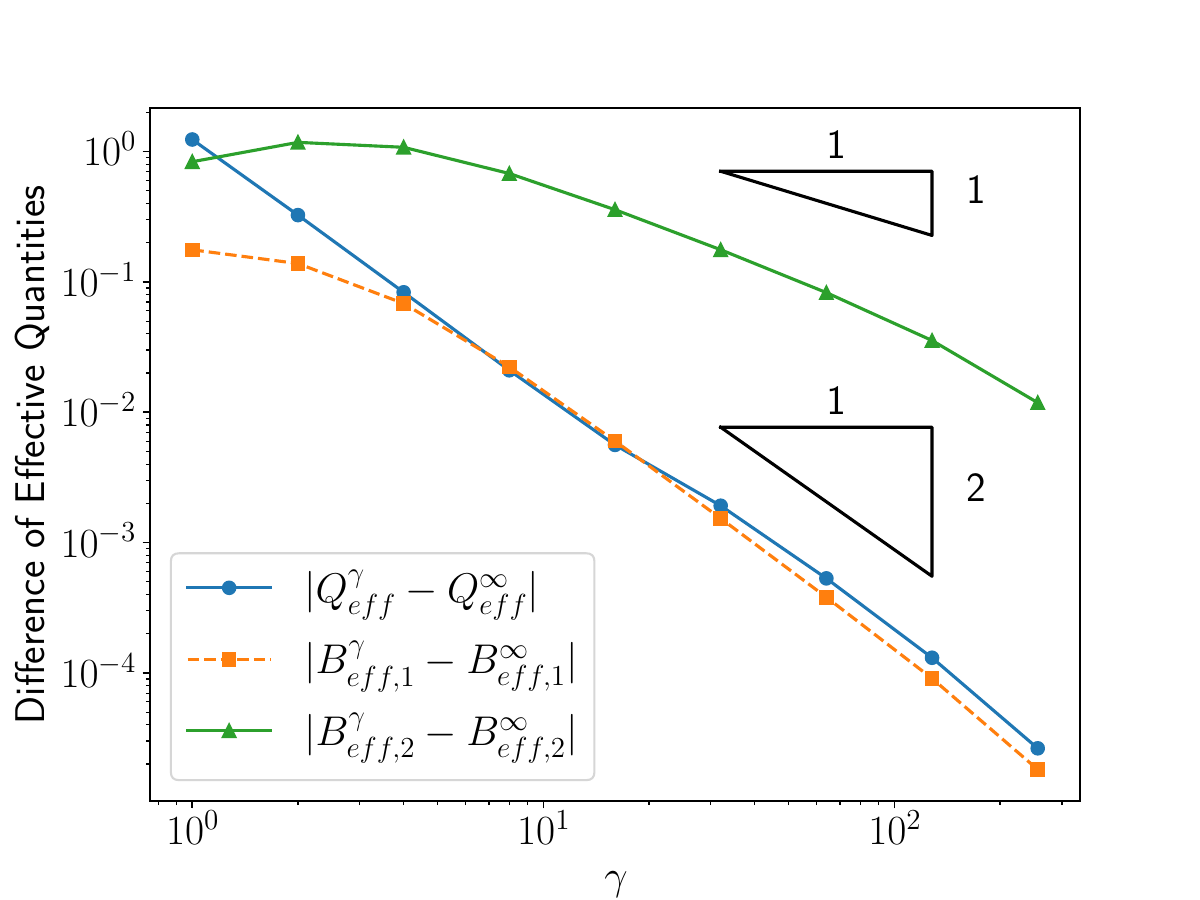}
\end{minipage}
\caption{Left: Periodic compliance tensor setup outlined in equations \eqref{eq:numerical-quadratic-form2},\eqref{eq:compliance-voigt}, and \eqref{eq:compliance-def}. Right: Difference of effective quantities as $\gamma\to\infty$. Notice that the effective stiffness converge like $\gamma^{-2}$. The effective prestrain converges like $\gamma^{-1}$ for $B_2$ and converges like $\gamma^{-1}$ for $B_1$. This shows that although Theorem \ref{T:quant2} does not apply in this example, we may expect a similar rate of convergence for the effective quantities for a larger class of quadratic forms.}
\label{fig:wood-example}
\end{figure}
 
\section{Introduction of the effective quantities}
\label{sect:Qeff}

In this section we present the definition of the effective quantities $Q^\gamma_{\eff}$, $B^\gamma_{\eff}$ and $\mathcal I_{\rm res}^\gamma$ for arbitrary $\gamma \in [0,\infty]$. The case $\gamma \in (0,\infty)$ is treated in \cite{bohnlein2023homogenized}. We recall the construction for the reader's convenience. The construction can be adapted to the extremal cases $\gamma \in \{0,\infty\}$ straightforwardly as outlined below.
The effective quadratic form $Q_{\eff}^\gamma$ is obtained by minimizing the original quadratic term $Q$ in the expansion of $W$ over a space of corrector fields, whose precise definition depends on $\gamma$.
\begin{definition}[Relaxation spaces]
Let $L^2(\mathcal{Y}, \R^{3\times 3}_{\rm sym})$ be the space of $\mathcal{Y}$-periodic functions
in $L^2_{\loc}(\R^2;\R^{3\times 3}_{\rm sym})$, let $Q$ be an admissible quadratic form. For each $x' \in S$ consider the Hilbert space
\begin{equation*}
  \HH
  \colonequals
  L^2\big(I, L^2(\mathcal{Y}, \R^{3\times 3}_{\rm sym})\big),
  \qquad
  \Big(H,H'\Big)_{x'}\colonequals\int_{\Box}\mathbb L (x', \placeholder) H :H' \,\dd (x_3,y).
\end{equation*}
Note that the induced norm satisfies
\begin{equation}
\label{def:innprod}
  \|H\|_{x'}^2=\int_{\Box}Q(x', x_3, y,H(x_3,y)) \,\dd (x_3,y).
\end{equation}
Moreover, for $\gamma \in [0,\infty]$ we introduce the subspaces
\begin{align}
    \HH^0_{\rm rel} 
    &\colonequals \begin{aligned}[t]
        \Big\{\iota(M)+ \sym (\iota(\nabla' \varphi + x_3\nabla^{'2}\zeta)+g\otimes e_3)\,:\,&M\in\R^{2\times 2}_{\sym},\, \zeta \in W^{2,2}(\mathcal{Y}) ,\, \\
        &\varphi \in W^{1,2}(\mathcal Y, \R^2),\, g \in L^{2}(I \times \mathcal Y, \R^{3})\Big\},
    \end{aligned} \\
    \HH^\gamma_{\rm rel} 
    &\colonequals \Big\{\iota(M)+ \sym (\nabla' \varphi | \tfrac{1}{\gamma}\partial_3\varphi)\,:\,M\in\R^{2\times 2}_{\sym},\,\varphi \in W^{1,2}(I \times \mathcal{Y}, \R^{3}))\Big\} \quad \text{if } \gamma \in (0,\infty), \\
    \HH^\infty_{\rm rel} 
    &\colonequals \Big\{\iota(M)+ \sym (\nabla' \varphi | d)\,:\,M\in\R^{2\times 2}_{\sym},\,\varphi \in L^{2}(I, W^{1,2}(\mathcal{Y}, \R^{3})) \, d \in L^{2}(I, \R^{3})\Big\}, 
\intertext{and}
    \HH^\gamma
    & \colonequals
    \Big\{\iota(x_3 G)+\chi\,:\,G\in\R^{2\times 2}_{\rm sym},\,\chi\in\HH^\gamma_{{\rm rel}}\Big\}.
\end{align}
Recall that
$$\iota:\R^{2\times 2}\to\R^{3\times 3},\qquad \iota(A):=\begin{pmatrix}
	A&\begin{array}{c}0\\0\end{array}\\\begin{array}{cc}0&0
	\end{array}&0
\end{pmatrix}.$$
We  denote by $\HH^{\gamma,\perp}_{{\rm rel}}$ the orthogonal
complement of $\HH^{\gamma}_{{\rm rel}}$ in $\HH^{\gamma}$ in the $(\cdot, \cdot)_{x'}$ inner product, and write $P^{\gamma,\perp}_{{\rm rel}}$
for the orthogonal projection from $\HH$ onto $\HH^{\gamma,\perp}_{{\rm rel}}$ in the $(\cdot, \cdot)_{x'}$ inner product.
Similarly, we write $P^{\gamma}$ for the orthogonal projection from $\HH$ onto $\HH^{\gamma}$ in the $(\cdot, \cdot)_{x'}$ inner product. 
\end{definition}

\begin{remark}
	We note that $\HH^\gamma_{\rm rel}$ is a closed subspaces of $\HH$, as can be seen by using Korn's inequality in combination with Poincar\'e's inequality. Also note that $\HH^{\gamma,\perp}_{{\rm rel}}, P^{\gamma,\perp}_{{\rm rel}}$ and $P^{\gamma}$ depend on $x'$. We neglect this dependence in the notation for simplicity.
\end{remark}

These relaxation spaces (also for $\gamma = 0$ and $\gamma = \infty$) have been introduced in \cite{neukamm2013derivation} in the context of the von Kármán regime. Using these spaces we introduce the homogenized quadratic forms $Q^\gamma_{\eff}$.

\begin{definition}[Homogenized quadratic form]\label{def:Qhom}
Let $Q$ be an admissible quadratic form in the sense of Definition~\ref{D:admissible} and let $\gamma \in [0,\infty]$. We define the associated homogenized quadratic form $Q^\gamma_{\ext}:S\times \HH\to[0,\infty)$ as
\begin{align}\label{eq:def:min}
    Q^\gamma_{\ext}(x',H) &\colonequals \inf_{\chi \in \HH^\gamma_{\rm rel}}\int_{\Box}Q\big(x', x_3,y,H+\chi \big)\dd (x_3,y),
\intertext{and define $Q^\gamma_{\eff}: S\times \R^{2 \times 2}_{\sym}\to[0,\infty)$ as}
    Q^\gamma_{\eff} (x',G) &\colonequals \inf_{\chi \in \HH^\gamma_{\rm rel}}\int_{\Box}Q\big(x', x_3,y,\iota(x_3 G)+\chi \big)\dd (x_3,y).
\end{align} 
\end{definition}
Before we can introduce the effective prestrain, we state the following lemma. 

\begin{lemma}\label{L:E}
Let $Q,B$ be admissible in the sense of Definition~\ref{D:admissible} and $\gamma \in [0,\infty]$. Then the map
\begin{equation*}
    \mathbf E^\gamma:S \times \R^{2\times 2}_{\sym}\to \HH^{\gamma,\perp}_{{\rm rel}},\qquad \mathbf E^\gamma(x',G)\colonequals P^{\gamma,\perp}_{{\rm rel}}\big(\iota(x_3G)\big)
\end{equation*}
is a linear isomorphism in the second variable, and
\begin{equation}
    \label{eq:ineq1'}
    \sqrt{\tfrac\alpha{12}}|G|\leq \bigg(\int_{\Box}|\mathbf E^\gamma(x',G)|^2 \dd (x_3,y) \bigg)^\frac12\leq \sqrt{\tfrac\beta{12}}|G|,
\end{equation}
for all $G\in \R^{2\times 2}_{\sym}$ and $x' \in S$. Furthermore, we have
\begin{equation*}
	Q^\gamma_{\eff} (x',G) = \|\mathbf E^\gamma(x',G)\|_{x'}^2,
\end{equation*}
and
\begin{equation}
    \label{eq:ineq2'}
    {\tfrac\alpha{12}}|G|^{2}\leq Q^{\gamma}_{\eff}(x',G) \leq {\tfrac\beta{12}}|G|^{2}.
\end{equation}
\end{lemma}

For $\gamma \in (0,\infty)$, the preceding lemma is \cite[Lemma 2.20]{bohnlein2023homogenized} and the proof for $\gamma = 0$ and $\gamma = \infty$ follow analogously. We give the proof in Appendix A.

\begin{definition}[Effective prestrain]\label{D:eff}
	Let $Q$ and $B$ be admissible in the sense of Definition~\ref{D:admissible}. We define the effective prestrain $B_{\eff}^\gamma \in L^{2}(S, \R^{2\times 2}_{\sym})$ associated with $Q$ and $B$ by
  \begin{equation}\label{eq:def:beff}
    B_{\eff}^\gamma(x') \colonequals(\mathbf E^\gamma(x'))^{-1}\Big(P^{\gamma,\perp}_{{\rm rel}}(\sym B(x', \cdot))\Big).
  \end{equation}
\end{definition}

Finally, having introduced $B^\gamma_{\eff}$ we can also introduce $\mathcal I^\gamma_{\rm res}$.

\begin{definition}[Residual energy]\label{D:Ires}
Let $Q$ and $B$ be admissible in the sense of Definition~\ref{D:admissible}. We define the \emph{residual energy} as
\begin{equation}
    \mathcal I_{\rm res}^\gamma \colonequals \int_{S}\int_{\Box}Q\Big(x',x_3,y,(\mathrm{Id}-P^\gamma)(\sym B(x',\cdot))\Big)\dd (x_3,y)\dd x'.
\end{equation}
\end{definition}

$Q^\gamma_{\eff}$ and $B^\gamma_{\eff}$ can also be characterized by corrector equations. For $\gamma \in (0,\infty)$, this is \cite[Lemma 2.23 and Proposition 2.25]{bohnlein2023homogenized}. In the appendix, we also show the cases $\gamma \in \{0,\infty\}$ which follows similarly.

\begin{lemma}[Existence of a corrector]\label{L:corrector}
Let $Q$ and $B$ be admissible in the sense of Definition~\ref{D:admissible} and fix $x' \in S$.

\begin{enumerate}[(a)]
  \item Let $H \in \HH^0$. There exists a unique quadruple
  \begin{equation}
      M_H^0\in\R^{2\times 2}_{\sym}, \qquad
      \zeta_H^0 \in W^{2,2}(\mathcal Y), \qquad
      \varphi_H^0 \in W^{1,2}(\mathcal{Y}, \R^{2}), \qquad
      g_H^0 \in L^{2}(I \times \mathcal Y, \R^{3}),
  \end{equation}
  with $\int_{\Box}\varphi_H^0\dd (x_3,y) =0$ and $\int_{\mathcal Y}\zeta_H^0\dd y =0$,
  solving the corrector problem
  \begin{equation}
      \int_{\Box}\mathbb L\left[H+\iota(M_H^0)+\sym(\iota(\nabla'\varphi_H^0 + x_3\nabla'^2\zeta_H^0)+g_H^0\otimes e_3)\right]: H' \,\dd (x_3,y)=0,%
    \footnote{For the rest of this section, we omit the dependence of $\mathbb L$ on $(x', x_{3}, y)$ for ease of notation.}
  \end{equation}
  for all $H' \in \HH^0_{\rm rel}$. Moreover, there exists a constant $C$ such that
      \begin{equation}
    |M_{H}^0|^{2} + \left\| \nabla' \varphi^0_{H} \right\|_{L^{2}}^{2} + \left\| \nabla'^2 \zeta^0_{H} \right\|_{L^{2}}^{2} + \|g_{H}^0\|_{L^{2}}^{2} \leq C \left\| H \right\|_{L^{2}}^{2}.
    \end{equation}
  We call $(M_H^0,\varphi_H^0, \zeta_H^0, g_H^0)$ the corrector associated with $H$ at 0.

  \item Let $\gamma \in (0,\infty)$ and $H \in \HH^\gamma$. There exists a unique pair
  \begin{equation}
      M_H^\gamma\in\R^{2\times 2}_{\sym}, \qquad
      \varphi_H^\gamma \in W^{1,2}(I \times \mathcal{Y}, \R^{3})
  \end{equation}
  with $\int_{\Box}\varphi_H^\gamma\dd (x_3,y) =0$, solving the corrector problem
  \begin{equation}
      \int_{\Box}\mathbb L\left[H+\iota(M_H^\gamma)+\sym(\nabla'\varphi_H^\gamma| \tfrac{1}{\gamma}\partial_3\varphi_H^\gamma)\right]:H'\,\dd (x_3,y)=0,
  \end{equation}
  for all $H' \in \HH^\gamma_{\rm rel}$. Moreover, there exists a constant $C$ such that
      \begin{equation}
    |M_{H}^\gamma|^{2} + \left\| (\nabla'\varphi_H^\gamma| \tfrac{1}{\gamma}\partial_3\varphi_H^\gamma) \right\|_{L^{2}}^{2} \leq C \left\| H \right\|_{L^{2}}^{2}.
    \end{equation}
  We call $(M_H^\gamma,\varphi_H^\gamma)$ the corrector associated with $H$ at $\gamma$.

  \item Let $H \in \HH^\infty$. There exists a unique triplet
  \begin{equation*}
    M_H^\infty\in\R^{2\times 2}_{\sym},\qquad 
    \varphi_H^\infty \in L^{2}(I, W^{1,2}(\mathcal{Y}, \R^{3})) \qquad 
    d_H^\infty \in L^{2}(I, \R^{3}),
  \end{equation*}
  with $\int_{\Box}\varphi_H\dd (x_3,y) =0$, solving the corrector problem 
  \begin{equation}\label{eq:corrector_equation}
    \int_{\Box}\mathbb L\left[H+\iota(M_H)+\sym(\nabla'\varphi_H | d_H)\right]:H'\,\dd (x_3,y)=0,
  \end{equation}
  for all $H' \in \HH^\infty_{\rm rel}$. Moreover, there exists a constant $C$ such that
  \begin{equation}\label{eq:corrector_apriori}
    |M_{H}|^{2} + \left\| \nabla' \varphi_{H} \right\|_{L^{2}}^{2} + \|d_{H}\|_{L^{2}}^{2} \leq C \left\| H \right\|_{L^{2}}^{2}.
  \end{equation}
  We call $(M_H,\varphi_H, d_H)$ the corrector associated with $H$ at infinity.
\end{enumerate}
\end{lemma}

\begin{proposition}[Representation via correctors]\label{P:1}
 Let $Q$ and $B$ be admissible in the sense of Definition~\ref{D:admissible}. Let
 $G_1$, $G_2$, $G_3$ be an orthonormal basis of $\R^{2\times 2}_{\sym}$ and $\gamma \in (0,\infty)$. Fix $x' \in S$ and denote by $(M_{i}^0,\varphi_{i}^0, \zeta_{i}^0, g_{i}^0)$, $(M_{i}^\gamma,\varphi_{i}^\gamma)$ and $(M_{i}^\infty,\varphi_{i}^\infty, d_{i}^\infty)$ the correctors associated with $\iota(x_3 G_i)$ in the sense of Lemma~\ref{L:corrector}.
  \begin{enumerate}
  \item[(a)] \label{P:1:b} (Representation of $Q^\gamma_\textnormal{hom}$).
    The matrices  $\widehat Q^0, \widehat Q^\gamma, \widehat Q^\infty \in\R^{3\times 3}$ defined by
    \begin{align*}
      \widehat Q_{ik}^0 &\colonequals\int_{\Box}\mathbb L\big(\iota(x_3 G_i+M_{i})+\sym(\nabla'\varphi_{i}^0 + x_3\nabla'^2\zeta_{i}^0 | g_{i}^0)\big):\iota(x_3 G_k)\dd (x_3,y), \\
      \widehat Q_{ik}^\gamma &\colonequals\int_{\Box}\mathbb L\big(\iota(x_3 G_i+M_{i}^\gamma)+\sym(\nabla'\varphi_{i}^\gamma | \tfrac{1}{\gamma}\partial_3\varphi_{i}^\gamma)\big):\iota(x_3 G_k)\dd (x_3,y), \\
      \widehat Q_{ik}^\infty &\colonequals\int_{\Box}\mathbb L\big(\iota(x_3 G_i+M_{i}^\infty)+\sym(\nabla'\varphi_{i}^\infty | d_{i}^\infty)\big):\iota(x_3 G_k)\dd (x_3,y),
    \end{align*}
    are symmetric and positive definite, and we have for each $\tilde \gamma \in [0,\infty]$ that
    \begin{equation}\label{P:1:coercivityQhat}
      \frac{\alpha}{12} \mathrm{Id}_{3\times 3}\leq \widehat Q^{\tilde \gamma}\leq \frac{\beta}{12} \mathrm{Id}_{3\times 3},
    \end{equation}
    in the sense of quadratic forms. Moreover, for all $G\in\R^{2\times 2}_{\rm sym}$ we have the representation
    \begin{equation}
    \label{eq:matrixrep}
      Q^{\tilde \gamma}_{\rm eff}(G)=\sum_{i,j=1}^3\widehat Q_{ij}^{\tilde \gamma}\widehat G_i\widehat G_j,
    \end{equation}
    where $\widehat{G}_1$, $\widehat{G}_2$, $\widehat{G}_3$ are the coefficients of $G$
    with respect to the basis $G_1$, $G_2$, $G_3$.

  \item[(b)]  \label{P:1:c} (Representation of $B^\infty_\textnormal{eff}$).
    Define $\widehat B^0, \widehat B^\gamma, \widehat B^\infty \in\R^3$ by
    \begin{align*}
      \widehat B_i^0 &\colonequals \int_{\Box}\mathbb L\big(\iota(x_3 G_i+M_{i}^0)+\sym(\nabla'\varphi_{i}^0 + x_3 \nabla'^2\zeta_{i}^0 | g_{i}^0)\big):B\dd (x_3,y), \\
      \widehat B_i^\gamma &\colonequals \int_{\Box}\mathbb L\big(\iota(x_3 G_i+M_{i}^\gamma)+\sym(\nabla'\varphi_{i}^\gamma | \tfrac{1}{\gamma}\partial_3\varphi_{i}^\gamma)\big):B\dd (x_3,y), \\
      \widehat B_i^\infty &\colonequals \int_{\Box}\mathbb L\big(\iota(x_3 G_i+M_{i}^\infty)+\sym(\nabla'\varphi_{i}^\infty | d_{i}^\infty)\big):B\dd (x_3,y), && i=1,2,3.
    \end{align*}
    Then, for $\tilde \gamma \in [0,\infty]$ we have $B_{\eff}^{\tilde \gamma}=\sum_{i=1}^3\big( (\widehat Q^{\tilde \gamma})^{-1}\widehat B^\gamma\big)_iG_i$.
  \end{enumerate}
\end{proposition}

\section{Proof of Theorem \ref{T:gammainfty}}
\label{sect:qual}

In this section we prove Theorem \ref{T:gammainfty}. \cite{neukamm2013derivation} deals with a similar problem, in the context of the von Kármán theory for non-prestrained plates. Although Theorem \ref{T:gammainfty} does not follow from any result in \cite{neukamm2013derivation}, the proof is an adaptation of their framework.  

\begin{proof}[Proof of Theorem \ref{T:gammainfty}]
We only present the argument for the extremal cases $\bar\gamma \in \{0, \infty\}$, since for $\bar\gamma \in (0, \infty)$ the statement is similar and easier. Throughout the proof $x'\in S$ is arbitrary but fixed from now on. Our argument relies on the corrector representation of Proposition~\ref{P:1} and the observation that the correctors converge as $\gamma$ converges to $0$ or $\infty$. The latter is shown via $\Gamma$-convergence in the first two steps.

\step 1 We start by defining the vector spaces
\begin{align*}
    &\mathbf{S}_{\gamma} \colonequals \left\{(M, \varphi) \in \R^{2\times 2}_{\sym} \times W^{1,2}(\Box ;\R^3) : \int_{\Box}\varphi \dd (x_3,y) =0 \right\},\\
    &\mathbf{S}_{\infty} \colonequals \left\{  (M,d, \varphi) \in \R^{2\times 2}_{\sym}\times L^{2}(I, \R^{3})  \times L^{2}(I, W^{1,2}(\mathcal{Y}, \R^{3})) : \int_{\Box}\varphi \dd (x_3,y) =0 \right\},\\
    &\mathbf{S}_{0} \colonequals \left\{(M,g, \varphi, \zeta) \in \R^{2\times 2}_{\sym} \times  L^{2}(\Box, \R^{3}) \times W^{1,2}(\mathcal{Y}, \R^{2}) \times W^{2,2}(\mathcal{Y}) : \int_{\Box}\varphi \dd (x_3,y) =0 , \int_{\mathcal Y}\zeta \dd y =0\right\}.
\end{align*}
Next, for an arbitrary but fixed $H \in \HH$, we define the energy functionals 
\begin{align*}
	&\mathcal{C}_{\gamma}: \mathbf{S}_{\gamma} \to \R,\\
	&\qquad\mathcal{C}_{\gamma} (M, \varphi) \colonequals \int_{\Box }Q\big(x', x_3,y,H+\iota(M) + \sym\nabla_\gamma\varphi\big)\dd (x_3,y),\\
	&\mathcal{C}_{\infty}: \mathbf{S}_{\infty} \to \R,\\
	&\qquad\mathcal{C}_{\infty} (M,d, \varphi) \colonequals \int_{\Box }Q\big(x', x_3,y,H+\iota(M)+\sym(\nabla' \varphi |d)\big)\dd (x_3,y),\\
	&\mathcal{C}_{0}: \mathbf{S}_{0} \to \R,\\
	&\qquad\mathcal{C}_{0} (M,g, \varphi, \zeta) \colonequals \int_{\Box }Q\big(x', x_3,y,H+\iota(M)+\sym(\iota(\nabla' \varphi + x_{3} \nabla^{2}\zeta)+g\otimes e_3)\big)\dd (x_3,y).
\end{align*}
For simplicity we drop the dependence of $Q$ on $x',x_3,y$ in the following.
We claim that, as $\gamma \to \bar\gamma\in\{0,\infty\}$, $\mathcal{C}_{\gamma}$ $\Gamma-$converges to $\mathcal{C}_{\bar\gamma} $ in the topology of weak $L^2$ convergence of $\sym\nabla_\gamma\varphi$, and convergence of $M$. \\ 

Note that the liminf inequality is immediate: if $\bar\gamma = \infty$, $M_{\gamma} \to M_{\infty}$, and $\sym\nabla_\gamma\varphi \rightharpoonup \sym(\nabla' \varphi_{\infty} |d_{\infty})$ then by positive semi-definiteness of $Q$,
\begin{equation}
\begin{split}
& \int_{\Box }Q\big(x', x_3,y,H+\iota(M_{\infty})+\sym(\nabla' \varphi_{\infty} |d_{\infty})\big)\dd (x_3,y)\\
\leq & \liminf_{\gamma \to \infty} \int_{\Box }Q\big(x', x_3,y,H+\iota(M_{\gamma})+\sym\nabla_\gamma \varphi_{\gamma}\big)\dd (x_3,y).
\end{split} 
\end{equation}
The proof of the liminf inequality in the case $\bar\gamma = 0$ is analogous and left to the reader. 

To prove the limsup inequality, we consider the cases $\bar\gamma=\infty$ and $\bar\gamma = 0$ separately.

\medskip\noindent\textit{Substep 1.1 } Limsup inequality in the case $\bar\gamma=\infty$. 

Let $(M_{\infty}, \varphi_{\infty}, d_{\infty}) \in \mathbf{S}_{\infty} $. Using a standard density argument, we may assume that $d_{\infty} \in C^{\infty} (I, \R^{3})$. Let 
\begin{equation}
D_{\infty}(x_{3}) \colonequals \int_{0}^{x_{3}} d_{\infty}(s) \dd s.
\end{equation}   

Consider the function 
\begin{equation}
\widetilde{\varphi}_{\gamma} \colonequals \varphi_{\infty} + \gamma D_{\infty},
\end{equation}
and the test function
\begin{equation}
    {\varphi}_{\gamma} \colonequals \widetilde{\varphi}_{\gamma} - \int_{\Box} \widetilde{\varphi}_{\gamma} \dd (x_3,y).  
\end{equation}
Then $(M_{\infty}, {\varphi}_{\gamma}) \in \mathbf{S}_{\gamma} $ and we can compute
\begin{equation}
\nabla_{\gamma} \varphi_{\gamma} = \left( \nabla' \varphi_{\infty} \bigg| \frac{1}{\gamma} \partial_{3} \varphi_{\infty} + d_{\infty} \right).
\end{equation}
This implies that
\begin{equation}
\begin{split}
\lim_{\gamma \to \infty} \mathcal{C}_{\gamma} (M_{\infty}, \varphi_{\gamma}) &= \lim_{\gamma \to \infty}  \int_{\Box }Q\left(x', x_3,y,H+\iota(M_{\infty})+\left( \nabla' \varphi_{\infty} \bigg| \frac{1}{\gamma} \partial_{3} \varphi_{\infty} + d_{\infty} \right)\right)\dd (x_3,y)\\
&= \int_{\Box }Q\left(x', x_3,y,H+\iota(M_{\infty})+\left( \nabla' \varphi_{\infty} \big| d_{\infty} \right)\right)\dd (x_3,y)\\
&= \mathcal{C}_{\infty} (M_{\infty},d_{\infty}, \varphi_{\infty}).
\end{split}
\end{equation}
On the other hand it is clear that $\nabla_{\gamma} \varphi_{\gamma} \rightharpoonup (\nabla' \varphi_{\infty}|  d_{\infty})$. This concludes substep 1.1. 

\medskip\noindent\textit{Substep 1.2 } Limsup inequality in the case $\bar\gamma=0$. 

Let $ (M_{0},g_{0}, \varphi_{0}, \zeta_{0}) \in \mathbf{S}_{0}$. Using a standard density argument, we may assume that $g_{0} \in C^{\infty} (\Box, \R^{3})$.

 Let 
\begin{equation}
G_{0}(x_{3},y) \colonequals \int_{0}^{x_{3}} g_{0}(s,y) \dd s.
\end{equation}   

Consider the function 
\begin{equation}
\widetilde{\varphi}_{\gamma} \colonequals (\varphi_{0},0) + x_{3} \nabla' \zeta_{0} - \frac{1}{\gamma} (0,0,\zeta_{0}) + \gamma G_{0},
\end{equation}
and the test function
\begin{equation}
    {\varphi}_{\gamma} \colonequals \widetilde{\varphi}_{\gamma} - \int_{\Box} \widetilde{\varphi}_{\gamma} \dd (x_3,y).  
\end{equation}
Then $(M_{0}, {\varphi}_{\gamma}) \in \mathbf{S}_{\gamma} $ and we can compute
\begin{equation}
\nabla_{\gamma} \varphi_{\gamma} =  \begin{pmatrix}
    \nabla' \varphi_{0} + x_{3} \nabla'^{2} \zeta_{0} + \gamma\nabla' G_{0} & \begin{pmatrix}
        g_{0}^{(1)} \\
        g_{0}^{(2)}
    \end{pmatrix}  \\
    0 & g_{0}^{(3)} \end{pmatrix} +
    \begin{pmatrix}
    0 & \frac{1}{\gamma} \nabla' \zeta_{0}  \\
    -\frac{1}{\gamma} \nabla' \zeta_{0}^{T} & 0 \end{pmatrix}.
\end{equation}
This implies that
\begin{equation}
\begin{split}
&\lim_{\gamma \to 0} \mathcal{C}_{\gamma} (M_{0}, \varphi_{\gamma}) \\
= & \lim_{\gamma \to 0}  \int_{\Box }Q\left(x', x_3,y,H+\iota(M_{0})+\sym\begin{pmatrix}
    \nabla' \varphi_{0} + x_{3} \nabla'^{2} \zeta_{0} + \gamma\nabla' G_{0} &  \begin{pmatrix}
        g_{0}^{(1)} \\
        g_{0}^{(2)}
    \end{pmatrix} \\
    0 & g_{0}^{(3)} \end{pmatrix} \right)\dd (x_3,y)\\
= & \int_{\Box }Q\left(x', x_3,y,H+\iota(M_{0})+\sym(\nabla' \varphi_{0} + x_{3} \nabla'^{2}\zeta_{0} |g_{0}) \right)\dd (x_3,y)\\
= & \mathcal{C}_{0} (M_{0},g_{0}, \varphi_{0}, \zeta_{0}).
\end{split}
\end{equation}
On the other hand it is clear that $\sym \nabla_{\gamma} \varphi_{\gamma} \rightharpoonup \sym(\nabla' \varphi_{0} + x_{3} \nabla'^{2}\zeta_{0} |g_{0})$. This concludes substep 1.2. 

\step 2 Given $x' \in S$ and $H \in \HH$, let $(M_{\gamma}, \varphi_{\gamma})$ denote the minimizer of $\mathcal{C}_{\gamma}$. We call $(M_{\gamma}, \varphi_{\gamma})$ the \emph{corrector} associated to $H$ at $\gamma$. Similarly, let $(M_{\infty}, \varphi_{\infty}, d_{\infty})$ denote the minimizer of $\mathcal{C}_{\infty}$, and $(M_{0},g_{0}, \varphi_{0}, \zeta_{0})$ denote the minimizer of $\mathcal{C}_{0}$. 

We claim that as $\gamma \to \infty$, $\sym \left( \nabla_{\gamma} \varphi_{\gamma} \right) + \iota(M_{\gamma}) \to \sym \left( \nabla'\varphi_{\infty} | d_{\infty} \right) + \iota(M_{\infty})$ and as $\gamma \to 0$, $\sym \left( \nabla_{\gamma} \varphi_{\gamma} \right) + \iota(M_{\gamma}) \to \sym(\nabla' \varphi_{0} |g_{0}) + \iota(x_{3} \nabla'^{2}\zeta_{0}) + \iota(M_{0})+ $ strongly in $L^{2}$. \\

To prove this claim, notice that the norm of the corrector is uniformly bounded for $\gamma \in (0, \infty)$:
\begin{equation}
\begin{split}
&\sup_{\gamma \in (0, \infty)} \left\| \sym \nabla_{\gamma} \varphi_{\gamma} \right\|_{L^{2}}^{2} + \left | M_{\gamma} \right|^{2} \\
= & \sup_{\gamma \in (0, \infty)} \int_{\Box } \left | \iota(M_{\gamma})+ \nabla_{\gamma} \varphi_{\gamma} \right|^{2} \dd (x_3,y) \\
\leq & \sup_{\gamma \in (0, \infty)} 2\int_{\Box } \left | H+\iota(M_{\gamma})+ \nabla_{\gamma} \varphi_{\gamma} \right|^{2} \dd (x_3,y) + 2\left\| H\right\|_{L^2}^2 \\
\leq &  \sup_{\gamma \in (0, \infty)} c \int_{\Box } Q\left(x', x_3,y,H+\iota(M_{\gamma})+ \nabla_{\gamma} \varphi_{\gamma} \right)\dd (x_3,y) + 2\left\| H\right\|_{L^2}^2 \\
\leq & c \int_{\Box } Q\left(x', x_3,y,H\right)\dd (x_3,y) + 2\left\| H\right\|_{L^2}^2 \\
\leq & C \|H\|_{L^{2}}^{2},
\end{split}
\end{equation} 
for some constants $c,C$. 

If $\gamma \to \infty$, then by \cite[Lemma 5.1]{neukamm2013derivation}, we have that  modulo a subsequence, $M_{\gamma} \to M $, and $\nabla_{\gamma} \varphi_{\gamma} \rightharpoonup \left( \nabla'\varphi| d \right)$ weakly in $L^{2}$, for some $(M,\varphi, d) \in \mathbf{S}_{\infty}$. 

By step 1 and by uniqueness of the minimizer (Lemma \ref{L:corrector}) of $\mathcal{C}_{\infty}$, $(M,\varphi, d) = (M_{\infty},\varphi_{\infty}, d_{\infty})$. Therefore $\sym \left( \nabla_{\gamma} \varphi_{\gamma} \right) + \iota(M_{\gamma}) \rightharpoonup \sym \left( \nabla'\varphi_{\infty} | d_{\infty} \right) + \iota(M_{\infty})$ weakly in $L^{2}$.

Furthermore, by $\Gamma-$convergence,
\begin{equation}
\begin{split}
&\int_{\Box }Q\left(x', x_3,y,H+\iota(M_{\infty})+\left( \nabla' \varphi_{\infty} \big| d_{\infty} \right)\right)\dd (x_3,y)\\
= &\lim_{\gamma \to \infty} \int_{\Box }Q\left(x', x_3,y,H+\iota(M_{\gamma})+ \nabla_{\gamma} \varphi_{\gamma} \right)\dd (x_3,y).
\end{split}
\end{equation}
Since $Q$ is positive definite on symmetric matrices, we conclude that $\sym \left( \nabla_{\gamma} \varphi_{\gamma} \right) + \iota(M_{\gamma}) \to \sym \left( \nabla'\varphi_{\infty} | d_{\infty} \right) + \iota(M_{\infty})$ strongly in $L^{2}$. 

Similarly, if $\gamma \to 0$, then by \cite[Lemma 5.1]{neukamm2013derivation}, we have that  modulo a subsequence, $M_{\gamma} \to M $, and $\nabla_{\gamma} \varphi_{\gamma} \rightharpoonup (\nabla' \varphi + x_{3} \nabla'^{2}\zeta |g)\big)$ weakly in $L^{2}$, for some $(M,g, \varphi, \zeta) \in \mathbf{S}_{0}$. The conclusion in the case $\gamma \to 0$ then follows by an analogous reasoning. 

\step 3 Conclusion. The convergence of the correctors shown in step 2 along with Proposition \ref{P:1} implies that $B^{\gamma}_{\eff} \to B^{\bar \gamma}_{\eff}$ and $Q^{\gamma}_{\eff}(G) \to Q^{\bar \gamma}_{\eff}(G)$. By Definitions \ref{def:Qhom} and \ref{D:eff}, 
\begin{equation}
    Q^{\gamma}_{\rm ext} \left( x',  B \left( x', x_{3}, y  \right) \right) = Q^{\gamma}_{\eff} \left( x', B_{\eff}^{\gamma} \right)\dd x' + \mathcal I_{\rm res}^\gamma.
\end{equation}
Furthermore, by steps $1$ and $2$, $\lim_{\gamma \to \bar \gamma} Q^{\gamma}_{\rm ext} \left( x',  B \left( x', x_{3}, y  \right) \right) = Q^{\bar \gamma}_{\rm ext} \left( x',  B \left( x', x_{3}, y  \right) \right)$, which in particular implies that $\mathcal I_{\rm res}^{\gamma} \to \mathcal I_{\rm res}^{\bar \gamma}$.
  
\end{proof}

\section{Proof of Theorem \ref{T1}}
\label{sect:proofT1}

In this Section, we will prove Theorem \ref{T1}. Before giving the proof, we state the exact hypotheses on the stored energy functional $W$. 
\begin{definition}[Stored energy function] \label{Def:StoredEnergyFct}
Let $\alpha > 0$ and  $W:S \times I \times \mathcal{Y} \times \R^{3\times 3} \to [0,\infty]$ be a Carathéodory function. We call $W$ a stored energy functional if $W$ is piecewise continuous in the second variable, continuous in the fourth variable and satisfies for a.e.~$x_3 \in \R$, $x' \in S$, and $y \in \mathcal{Y}$ the following statements:
\begin{enumerate}
\renewcommand{\labelenumi}{\theenumi}
\renewcommand{\theenumi}{(W\arabic{enumi})}
	\item \label{Ass:FrameIndifference}(Frame indifference). $W(x', x_3, y,RF) = W(x', x_3,y,F)$ for all $F \in \R^{3 \times 3}$ and $R \in \SO{3}$.
	\item \label{Ass:NaturalState}(Natural state). $0 = W(x', x_3, y, \mathrm{Id}) \leq W(x', x_3, y, F)$ for all $F \in \R^{3 \times 3}$.
	\item \label{Ass:NonDegeneracy}(Non-degeneracy). $W(x', x_3, y, F) \geq \tfrac{1}{\alpha} \dist^2(F, \SO{3})$ for all $F \in \R^{3 \times 3}$.
	\item \label{Ass:Expansion}(Quadratic expansion). There exists a Carathéodory function $Q: S \times I \times \mathcal{Y} \times \R^{3 \times 3} \to [0,\infty]$ which is piecewise continuous in the second variable, and a non-negative quadratic form in the fourth variable and satisfies the uniform bound
	\begin{equation}
		Q(x', x_3,y, G) \leq \alpha |G|^2.
	\end{equation}
	Moreover, there exists a Carathéodory function $r: S \times \R \times \R^3 \times [0,\infty) \to [0,\infty]$, which is piecewise continuous in the second variable, continuous and increasing in the fourth variable and satisfies for each $x' \in S, x_{3} \in I$ that
    \begin{equation}
        \lim_{t \to 0} r(x', x_{3}, y,t) = 0
    \end{equation}
    and
    \begin{equation}
        \limsup_{t \to 0} \operatorname*{ess\,sup}_{ y \in \mathcal{Y}} r(x', x_{3}, y,t) < \infty,
    \end{equation}
    such that
	\begin{equation}\label{Eq:TaylorExpansion}
		\left|W(x', x_3, y, \mathrm{Id}+G) - Q(x', x_3,y,G)\right| \leq |G|^2 r(x', x_3, y, |G|) \quad \text{ for a.e.~} y \in \mathcal{Y}.
	\end{equation}
	\item (growth conditions). For all $F,G \in \R^{3\times 3}$, we have
	\begin{equation}\label{Eq:growth_cond}
		\begin{cases}
			\tfrac{1}{\alpha} |F|^2 - \alpha \leq W(x', x_3,y,F), \\ 
			|W(x', x_3,y,F) - W(x', x_3,y,G)| \leq \alpha(1 + |F| + |G|) |F - G|.
		\end{cases}
	\end{equation}
	Furthermore, there exists a neighborhood $O$ of $ \SO{3}$ such that for all $F \in O$,
	\begin{equation}
	 W(x', x_3,y,F) \leq \alpha (1 + |F|^2)
	\end{equation}
\end{enumerate}
\end{definition}

\begin{remark}\label{Rem:Properties_W_and_Q}
It is well-known in elasticity, that \hyperref[Def:StoredEnergyFct]{(W1) - (W4)} imply
\begin{equation}\label{Eq:UniformBounds_Q}
		\tfrac{1}{\alpha} |\sym G|^2 \leq Q(x', x_3, y,G) \leq \alpha|\sym G|^2,
\end{equation}
for all $G \in \R^{3\times 3}$ and a.e.~$x_3 \in I$, $x' \in S$, and $y \in \R^3$. In particular, $Q$ only depends on the symmetric part of $G$.
\end{remark}

Next, we recall the main theorem of \cite{neukamm2025linearization}, which we will need in the proof of Theorem \ref{T1}. 
\begin{theorem}[{\cite[Theorem 3.2]{neukamm2025linearization}}]
\label{Thm:Expansion_Whom}
Let $W$ satisfy see Definition \ref{Def:StoredEnergyFct}, and $W^h_\ho$ be given by equation \eqref{eq:formW}. Then the following statements hold.
\begin{enumerate}
	\item[(a)] (Frame indifference). For all $F \in \R^{3 \times 3}$, $h > 0$, $R \in \SO{3}$, and a.e.~$x_3 \in I$, $x' \in S$,  we have
	\begin{equation} \label{Eq:FrameIndifference_Whom}
		W^h_\ho(x', x_3, RF) = W^h_\ho(x', x_3, F).
	\end{equation}
	\item[(b)] (Non-degeneracy). There exists some $\alpha' > 0$ and $h_0 > 0$ such that for all $F \in \R^{3 \times 3}$,  $0 < h \leq h_0$, and a.e.~$x_3 \in I$, $x' \in S$, 
	\begin{equation} \label{Eq:NonDegeneracy_Whom}
		W^h_\ho(x', x_3, F) \geq \tfrac{1}{\alpha'} \dist^2(F, \SO{3}) - \alpha'h^2.
	\end{equation}
	\item[(c)] (Quadratic expansion). There exists a Carathéodory function $\rho: S \times \R \times [0,\infty) \to [0,\infty]$ which is continuous and increasing in the third variable, such that for a.e.~$x_3 \in \R, x' \in S$ we have $\rho(x', x_3,0) = 0$ and for all $h > 0$ and $G \in \R^{3 \times 3}$
	\begin{equation}\label{Eq:TaylorExpansion_Whom}
	\begin{split}
		&\left|\frac{1}{h^2} W^h_\ho(x', x_3,\mathrm{Id} + hG) - \left(Q_{\ho}(x', x_3,G - B_\ho(x', x_3)) + R(B)(x', x_3)\right)\right|\\
		 \leq & (1 + |G|^2) \rho(x', x_3, h + |hG|).
	\end{split}	
	\end{equation}
\end{enumerate}
\end{theorem}

Here, $Q_{\ho}$ is defined as
\begin{equation}
	Q_{\ho}(x', x_3, G) := \inf_{\varphi} \int_Y Q(x', x_3, y, G + \iota (\sym \nabla' \varphi) ) \,\dd y,
\end{equation}
where the infimum is taken over $\varphi \in L^{2}(I, W^{1,2}(\mathcal{Y}, \R^{3}))$; and $B_\ho$ and $R(B)$ are defined as follows: Consider the scalar product
\begin{equation}
	\left( \psi, \chi \right)_{x} := \int_Y  \mathbb L(x', x_3,y) \psi(y) : \chi(y) \,\dd y,
\end{equation}
on $L^2(\mathcal Y,\R^{3\times 3}_{\sym})$ and the closed subspaces 
\begin{align*}
	\mathbf{S} := \{\iota(\sym(\nabla' {\varphi}))\mid \varphi \in L^{2}(I, W^{1,2}(\mathcal{Y}, \R^{3})) \}, &&
	\mathbf{C} := \R^{3\times 3}_{\sym}, &&
	\mathbf{O} := \mathbf{S}^{\perp_{\mathbf{S} + \mathbf{C}}} = \mathbf{S}^\perp \cap (\mathbf{S} + \mathbf{C}).
\end{align*}
Define $P_{x}:\R^{3\times 3}_{\sym} \to \mathbf{O}$ as the orthogonal projection onto $\mathbf{O}$ in the $(\cdot, \cdot)_{x}$ inner product, restricted to $\R^{3\times 3}$. It turns out that $P_{x}$ is a linear isomorphism, see \cite{neukamm2025linearization}. $B_\ho$ and $R(B)$ are defined as:
\begin{equation}
\label{eq:BhomRB}
\begin{split}
	B_{\hom}(x', x_3) &:= P_{x}^{-1}(\operatorname{P}_\mathbf{O}(\sym B(x', x_{3}, \cdot))), \\
	R(B)(x', x_3) &:= \|\operatorname{P}_{(\mathbf{S} + \mathbf{C})^\perp}(\sym B(x', x_{3}, \cdot))\|_{x}^2,
\end{split}
\end{equation}
where $\operatorname{P}_{(\mathbf{S} + \mathbf{C})^\perp}$ and $\operatorname{P}_\mathbf{O}$ denote, respectively, the orthogonal projection onto $(\mathbf{S} + \mathbf{C})^\perp$ and onto $\mathbf{O}$ in the $(\cdot, \cdot)_{x}$ inner product. 

\begin{remark}
    It would be reasonable to work with a slightly stronger version of assumption ($W4$), namely that the error in the quadratic expansion is uniform in $x'$ and $x_{3}$. More precisely, it would be reasonable to assume that $r$ satisfies that $\limsup_{t \to 0} \operatorname*{ess\,sup}_{x' \in S, x_{3} \in I, y \in \mathcal{Y}} r(x', x_{3}, y,t) < \infty$. Under this stronger hypothesis, it would be possible to show  an analogue of item (c) of Theorem \ref{Thm:Expansion_Whom} in which $\rho$ is independent of $x'$ and $x_{3}$. Our proof would simplify in that case. Our methods, however, allow us to treat a $\rho$ which depends on $x'$ and $x_{3}$. 
\end{remark}

\begin{remark}
    The results in \cite{neukamm2025linearization} are actually more general than the setting of Theorem \ref{Thm:Expansion_Whom}. Among other things, \cite{neukamm2025linearization} can treat, in addition a dependence on $\frac{x_{3}}{\epsilon}$. The reason for not including this additional dependence in the hypotheses of Theorem \ref{Thm:Expansion_Whom} is that we want to be able to compare the result (after letting $h$ tend to $0$) with that of \cite{bohnlein2023homogenized,neukamm2013derivation}, which do not deal with a dependence on $\frac{x_{3}}{\epsilon}$.
\end{remark}

We proceed to give the proof of Theorem \ref{T1}. After proving item (a), we will prove a different version of items (b) and (c). That is, we will prove the analogues of items (b) and (c) for a different energy, and then show that the two energies are equivalent at the end of this Section. The energy we will work with in this section is defined as follows:
\begin{definition}
We define the functional $ \widetilde{\mathcal{E}}^{\infty}_{\hom}$, as
\begin{equation}
  \label{eq:Etilde}
 \widetilde{\mathcal{E}}^{\infty}_{\hom}(\deform)\colonequals
  \begin{cases}
    \displaystyle
    \min_{M,d}\int_{\Omega} Q_{\ho}(x',x_3,\iota(x_3\II_{\deform}+M)+d \otimes e_{3} -B_\ho(x) ) + \mathrlap{R(B)(x)  \dd x} \\
    &  \text{if }\deform\in W^{2,2}_{\iso}(S;\R^3),\\
    +\infty & \text{otherwise.}
  \end{cases}
\end{equation}
  where $e_{3}\colonequals (0,0,1) \in \R^3$ and the minimization is over all pairs $(M,d)$ with $M\in L^2(S;\R^{2\times 2}_{\sym})$ and $d\in L^2(\Omega;\R^{3})$, $\II_{\deform}$ denotes the second fundamental form associated to $\deform$, and $B_\ho, R(B)$ are given by equation \eqref{eq:BhomRB}. 
\end{definition}

\begin{remark}
 We note that we may rewrite the 2d homogenized energy $\widetilde{\mathcal{E}}^{\infty}_{\hom}$ as 
\begin{equation}
\label{eq:Etildeequiv}
  \widetilde{\mathcal{E}}^{\infty}_{\hom}(\deform)\colonequals
  \begin{cases}
    \displaystyle \int_{\Omega} Q^{\infty}_{\rm ext} \left( x', \iota(x_{3}\II_{\deform}) - B \left( x', x_{3}, y  \right) \right)\dd x'&\text{if }\deform\in W^{2,2}_{\iso}(S;\R^3),\\
    +\infty&\text{otherwise,}
  \end{cases}
\end{equation}
with $Q^{\infty}_{\rm ext} $ given by Definition \ref{def:Qhom}. 

In order to prove this claim, note that for each $x' \in S$, 
\begin{equation}
    \begin{split}
        &Q^{\infty}_{\rm ext} \left( x', \iota(x_{3} \II_{\deform}) - B \left( x', \cdot \right) \right)\\
        =& \inf_{M,\varphi,d}\int_{\Box}Q\big(x', x_3,y, \iota(x_{3} \II_{\deform}) - B \left( x', x_{3}, y \right)+\iota(M)+\sym (\nabla' \varphi | d) \big)\dd (x_3,y)\\
        =& \inf_{M,d} \inf_{\varphi}  \int_{\Box}Q\big(x', x_3,y,\iota(x_{3} \II_{\deform}) - B \left( x', x_{3}, y \right)+\iota(M)+\sym (\nabla' \varphi | d) \big)\dd (x_3,y)\\
        =& \inf_{M,d} \int_{I} Q_{\ho}(x',x_3,\iota(x_3\II_{\deform}+M)+d \otimes e_{3} - B_\ho(x', x_{3}) ) + R(B)(x', x_{3}) \dd x_3.
    \end{split}
\end{equation}
Above, the infimum is taken over $ M \in \R^{2\times 2}_{\sym} , \varphi \in L^{2}(I, W^{1,2}(\mathcal{Y}, \R^{3})), d \in L^{2}(I, \R^{3})$. 
\end{remark}

We now state the exact Theorem we will prove, which we will then show to be equivalent to Theorem~\ref{T1}. 
\begin{theorem}[$\Gamma$-convergence]\label{T1'}
  Let $  \mathcal{E}_{h, \hom}$ be given by equation \eqref{eq:homenergy}. Then the following statements hold:
  \begin{enumerate}[(\alph*)]
  \item[(a)] \label{item:T1:compactness'}
   (Compactness). Let $(\deform_h)\subset L^2(\Omega;\R^3)$ be a sequence with equibounded energy, i.e.,
    \begin{equation}\label{eq:equibounded'}
      \limsup_{h\to0}  \mathcal{E}_{h, \hom}(\deform_h)<\infty.
    \end{equation}
    Then there exists $\deform\in W^{2,2}_{\iso}(S;\R^3)$ and a subsequence (not relabeled) such that
    \begin{subequations}\label{T1:conv'}
      \begin{alignat}{2}
      \deform_h-\fint_{\Omega}\deform_h\dd x & \to \deform & \qquad & \text{in $L^2(\Omega)$},\\
      \nabla_h \deform_h & \to (\nabla'\deform,b_y) && \text{in $L^2(\Omega)$}.
    \end{alignat}
    \end{subequations}
  \item[(b)] \label{item:T1:lower_bound'} (Lower bound). If $(\deform_h)\subset L^2(\Omega;\R^3)$ is a sequence with $\deform_h-\fint_{\Omega}\deform_h\dd x\to y$ in $L^2(\Omega)$, then
    \begin{equation*}
      \liminf_{h\to0}  \mathcal{E}_{h, \hom}(\deform_h)\geq \widetilde{\mathcal E}^{\infty}_{\hom} (\deform).
    \end{equation*}
  \item[(c)] \label{item:T1:recovery_sequence'} (Recovery sequence). For any $\deform\in W^{2,2}_{\iso}(S;\R^3)$ there exists a sequence $(\deform_h)\subset W^{1,\infty}(\Omega;\R^3)$ with $\deform_h\to y$ strongly in $W^{1,2}(\Omega;\R^3)$ such that
    \begin{equation}\label{T1:c'}
      \lim_{h\to0}  \mathcal{E}_{h, \hom}(\deform_h)=\widetilde{\mathcal E}^{\infty}_{\hom} (\deform).
    \end{equation}
  \end{enumerate}
\end{theorem}%

The derivation of a nonlinear bending theory for prestrained thin elastic plates is the subject of \cite{padilla2022dimension}. However, in our case the energy functional is not of the form 
\begin{equation}
	\frac{1}{h^{2}} \int_{\Omega} W(x',x_3, \nabla_{h}\deform(x)(\mathrm{Id}+hB(x',x_3))^{-1}) \,\dd x.
\end{equation}
Therefore Theorem \ref{T1'} does not follow from \cite{padilla2022dimension}. 

We start by giving the short proof of the compactness statement of Theorem~\ref{T1'}. 

\begin{proof}[Proof of Theorem~\ref{T1'} item (a)\label{SS:compactness}]
By item (b) of Theorem \ref{Thm:Expansion_Whom}, there holds
    \begin{equation}\label{eq:FBE'}
      \limsup_{h\to0}\frac{1}{h^2}\int_\Omega \dist^2\big(\nabla_h\deform_h(x),\SO 3\big)\dd x<\infty.
    \end{equation}

Equation \eqref{T1:conv'} then follows from equation \eqref{eq:FBE'} by Theorem 4.1 of \cite{FJM02}.    
\end{proof}

We move on to prove the other items in  Theorem~\ref{T1'}.  

\subsection{Lower bound: Proof of Theorem~\ref{T1'} item (b)}\label{SS:lowerbound}

This proof will require the following lemma, which is an extension of Egorov's Theorem. 

\begin{lemma}
\label{lem:egorov}
Let $f_{n}: \Omega \to \R$ be a sequence of measurable functions converging pointwise a.e. to $f: \Omega \to \R$. Then there exists a sequence of measurable sets $D_{n} \subset \Omega$ such that 
\begin{itemize}[(\alph*)]
\item[(a)]  $ \lim_{n \to \infty}\left| \Omega \setminus D_{n} \right| =0 $, where $|\cdot|$ denotes the Lebesgue measure.

\item[(b)] $f_{n} \mathbf{1}_{D_{n}}$ converges uniformly to $f$.
\end{itemize}
\end{lemma}

\begin{proof}
We will only sketch the proof, since it is elementary. For each $\epsilon >0$, by Egorov's Theorem there exists a set $\widetilde{D}_{\epsilon}$ such that 
\begin{itemize}[(\alph*)]
\item[(a)]  $ \left| \Omega \setminus \widetilde{D}_{\epsilon} \right| \leq \epsilon $.

\item[(b)] $f_{n} \mathbf{1}_{\widetilde{D}_{\epsilon}}$ converges uniformly to $f$.
\end{itemize}
By reparametrizing, slowing down (or more precisely, applying a retardation argument as in \cite[Lemma 7]{padilla2020asymptotic} to) the sequence $\widetilde{D}_{1}, \widetilde{D}_{\frac{1}{2}}, \widetilde{D}_{\frac{1}{3}}, ...$, we obtain the desired sequence. 

\end{proof}

Equipped with Lemma, \ref{lem:egorov}, we will now give the proof of Theorem~\ref{T1'} item (b). 

\begin{proof}[Proof of Theorem~\ref{T1'} item (b)]

Proceeding as in \cite{FJM02}, we have that for each $h>0$ there exists a rotation-valued field $R_{h} \in L^{\infty} (S , \SO 3)$ (independent of $x_{3}$) such that
\begin{equation}
\label{eq:h2bound}
\limsup_{h \to 0} \frac{1}{h^{2}} \int_{\Omega} \left| \nabla_{h} v_{h} - R_{h} \right|^{2} \dd x < \infty.
\end{equation}

We now define the nonlinear strain
\begin{equation*}
E_{h}(v_{h}) \colonequals  \frac{ R_{h}^{T}  \nabla_{h}v_{h}- \mathrm{Id}}{h}.
\end{equation*}
By equation \eqref{eq:h2bound}, we have that
\begin{equation}
\label{eq:L2bound}
\limsup_{h \to 0} \left\| E_{h}(v_{h}) \right\|_{L^{2}} < \infty,
\end{equation}
and therefore (modulo a subsequence, not relabeled) $E_{h}(v_{h}) \rightharpoonup E$ weakly in $L^{2}$, for some $E \in L^{2}(\Omega, \R^{3 \times 3})$. We also define the set
\begin{equation*}
g \Omega_{h} \colonequals \{x \in \Omega : \left| E_{h}(v_{h})(x) \right| < h^{-\frac{1}{2}} \}, 
\end{equation*}
and the function
\begin{equation*}
\chi_{h} \colonequals \mathbf{1}_{g \Omega_{h}}.
\end{equation*}
Note that equation \eqref{eq:L2bound} implies that $\chi_{h} \to \mathbf{1}_{\Omega}$ pointwise a.e., and therefore in $L^{p}$ for every $p \in [1, \infty)$. Let $\rho$ be as in item (c) of Theorem \ref{Thm:Expansion_Whom}, and define the sequence of functions $\varphi_{h}: \Omega \to \R^{+}$ as 
\begin{equation}
\varphi_{h}(x',x_3) \colonequals \rho(x', x_{3}, h + h^{\frac{1}{2}}). 
\end{equation}
Note that $\varphi_{h}$ converges to the function $0$ pointwise a.e. as $h \to 0$. Therefore by Lemma \ref{lem:egorov}, there exists a sequence of measurable sets $D_{h} \subset \Omega$ such that 
\begin{enumerate}[(\alph*)]
\item[(a)]  $ \lim_{h \to 0}\left| \Omega \setminus D_{h} \right| =0 $.

\item[(b)] $\varphi_{h} \mathbf{1}_{D_{h}}$ converges uniformly to $0$.
\end{enumerate}

We define the sequence of functions $\xi_{h}: \Omega \to \R$ as
\begin{equation*}
\xi_{h} \colonequals \mathbf{1}_{D_{h}},
\end{equation*}
and also the function $\overline{\rho} : \R \to \R$
\begin{equation}
\overline{\rho} (h) = \sup_{x \in \Omega} \xi_{h}(x)\varphi_{h}(x).
\end{equation}
Note that $ \lim_{h \to 0} \overline{\rho} (h)=0$ by definition of $D_{h}$. 

We can now write
\begin{equation*}
\begin{split}
&\int_{\Omega} \left| \frac1{h^2} \xi_{h}\chi_{h}W^h_\ho \Big(x, \mathrm{Id} +hE_{h}(v_{h}) \Big) - Q_{\ho}(x, \xi_{h}\chi_{h}( E_{h}(v_{h}) - B_\ho) ) - R(B) \right| \dd x \\
 \leq&\int_{D_{h} \cap g \Omega_{h}} \left(1 + \left| \xi_{h}\chi_{h}E_{h}(v_{h}) \right|^{2} \right) \rho\left(x, h+ |hE_{h}(v_{h})|\right) \dd x \\
 \leq&\int_{\Omega} \varphi_{h}(x) \dd x + \int_{\Omega} \xi_{h}\chi_{h}|E_{h}(v_{h})|^{2} \varphi_{h}(x) \dd x \\
\leq&\int_{\Omega} \varphi_{h}(x) \dd x + \overline{\rho}(h) \int_{\Omega}|E_{h}(v_{h})|^{2} \dd x.
\end{split}
\end{equation*}

Note that 
\begin{equation}
\lim_{h \to 0} \int_{\Omega} \varphi_{h}(x) \dd x =0
\end{equation}
by Monotone Convergence Theorem since $\varphi_{h}(x)$ converges monotonically to $0$ for each $x$, and also that 
\begin{equation}
\lim_{h \to 0}  \overline{\rho}(h) \int_{\Omega}|E_{h}(v_{h})|^{2} \dd x =0
\end{equation}
since $\lim_{h \to 0} \overline{\rho}(h) =0$ and $\limsup_{h \to 0} \left\| E_{h}(v_{h}) \right\|_{L^{2}} < \infty$. 

Therefore, we conclude that

\begin{equation}
\begin{split}
&\liminf_{h \to 0}\int_{\Omega}\frac1{h^2} W^h_\ho \Big(x,\nabla_h v_{h} \Big)  \dd x \\
=&\liminf_{h \to 0}\int_{\Omega}\frac1{h^2} W^h_\ho \Big(x, \mathrm{Id} +hE_{h}(v_{h}) \Big) \dd x \\
 \geq &\liminf_{h \to 0}\int_{\Omega}\frac1{h^2} \xi_{h}\chi_{h}W^h_\ho \Big(x, \mathrm{Id} +hE_{h}(v_{h}) \Big) \dd x \\
 = &\liminf_{h \to 0} \int_{\Omega} Q_{\ho}(x, \xi_{h}\chi_{h}(E_{h}(v_{h}) - B_\ho) ) + \xi_{h}\chi_{h} R(B) \dd x. 
\end{split}
\end{equation}

Note that $\xi_{h}\chi_{h}$ converges to $\mathbf{1}_{\Omega}$ weakly in $L^{2}$. Hence, by weak-strong lemma, and the lower semi-continuity of $Q_{\ho}$, we infer that 
\begin{equation}
\liminf_{h \to 0} \int_{\Omega} Q_{\ho}(x, \xi_{h}\chi_{h}(E_{h}(v_{h}) - B_\ho) ) + \xi_{h}\chi_{h} R(B) \dd x  \geq \int_{\Omega} Q_{\ho}(x, E - B_\ho) ) + R(B) \dd x.
\end{equation}

Furthermore, using results from \cite[Proof of Theorem 6.1(i)]{FJM02}, we deduce that
  \begin{equation}
\sym E= \iota(x_3\II_{\deform}+M)+ \sym (d \otimes e_{3}),
  \end{equation}
for some $M\in L^2(S;\R^{2\times 2}_{\sym})$ and $d\in L^2(\Omega;\R^{3})$. Then, by definition of $\widetilde{\mathcal{E}}^{\infty}_{\hom}$ (equation \eqref{eq:Etilde}), we have that
\begin{equation}
\int_{\Omega} Q_{\ho}(x, E - B_\ho) ) + R(B) \dd x \geq \widetilde{\mathcal{E}}^{\infty}_{\hom}.
\end{equation}
From this, we can conclude. 
\end{proof}

\subsection{Upper bound: Proof of Theorem~\ref{T1'} item (c)}\label{SS:upperbound}

\begin{proof}[Proof of Theorem~\ref{T1'} item (c)]

Using a standard density argument we may assume that $\deform\in W^{2,2}_{\iso}(S, \mathbb{R}^{3}) \cap C^{\infty}(\overline{S}, \mathbb{R}^{3})$.  We define $M^{*} \in L^{2}(S, \mathbb{R}^{2 \times 2}_{\sym}) $ and  $d^{*} \in L^{2}(\Omega, \mathbb{R}^{3})$ as
\begin{equation}
(M^{*}, d^{*}) \colonequals {\rm argmin}_{M,d}\int_{\Omega} Q_{\ho}(x',x_3,\iota(x_3\II_{\deform}+M)+d \otimes e_{3} -B_\ho)\dd x.
\end{equation}
Note that $(M^{*}, d^{*}) $ exist, since they are the solution to a quadratic minimization problem on a linear space. Note also that the eigenvalues of $M^{*}$ are bounded below uniformly in $x'$, since $\II_{\deform}, Q_{\ho}, B_{\hom}$ are uniformly bounded. We may then apply Proposition 4.4 of \cite{bartels2023nonlinear}, and infer the existence of a sequence $v_{h} \in C^{\infty}(\Omega, \mathbb{R}^{3})$ such that:
\begin{itemize}
\item[(a)] 
\begin{equation}
\limsup_{h \to 0} \| v_{h} - \deform\|_{L^{\infty}} = 0\footnote{In this equation, we are identifying $\deform$ with its trivial extension to $\Omega$}.
\end{equation}  

\item[(b)] 
\begin{equation}\label{eq:strongconvergence}
\limsup_{h \to 0} \| \widetilde{E}_{h}(v_{h}) - \left(  \iota (x_{3}\II_{\deform} + M^{*}) + \sym(d^{*} \otimes e_{3} ) \right) \|_{L^{2}} = 0,
\end{equation}
where
\begin{equation}
\widetilde{E}_{h}(v_{h}) \colonequals \frac{\sqrt{ \nabla_{h} v_{h}^{T}\nabla_{h} v_{h}} - \mathrm{Id}}{h}.
\end{equation}

\item[(c)] 
\begin{equation}
\limsup_{h \to 0} h^{- \beta} \| \nabla_{h}v_{h} - R_{\deform} \|_{L^{\infty}} = 0
\end{equation}
 for some given $\beta \in (0,1)$, where
\begin{equation}
R_{\deform} \colonequals \left(\nabla' \deform,b_\deform\right).
\end{equation}
\end{itemize}

We now claim that $v_{h}$ satisfies 
    \begin{equation}
    \label{eq:limitenergy}
      \lim_{h\to0}  \mathcal{E}_{h, \hom}(\deform_h)=\widetilde{\mathcal E}^{\infty}_{\hom} (\deform).
    \end{equation}
The rest of this proof section will be aimed at proving equation \eqref{eq:limitenergy}. 

First, note that by polar factorization, there exists a sequence of measurable fields $R_{h} \in L^{\infty}( \Omega, \SO 3)$ such that
\begin{equation}
\nabla_{h} v_{h} = R_{h} \left( \mathrm{Id} + h \widetilde{E}_{h}(v_{h})  \right).
\end{equation}
Note also that item (c) implies that
\begin{equation}
\lim_{h \to 0} \| h \widetilde{E}_{h}(v_{h})  \|_{L^{\infty}} = 0.
\end{equation} 

We will now prove that
\begin{equation}
\begin{split}
\label{eq:upboundint}
&\lim_{h \to 0} \Bigg| \int_{\Omega}\frac1{h^2} W^h_\ho \Big(x, \mathrm{Id} +h\widetilde{E}_{h}(v_{h}) \Big) \dd x \\
&\ \ \ \ - \int_{\Omega} Q_{\ho}(x',x_3,\iota(x_3\II_{\deform}+M^{*})+\sym(d^{*} \otimes e_{3}) -B_\ho) + R(B)\dd x \Bigg|=0,
\end{split}
\end{equation}
which is equivalent to equation \eqref{eq:limitenergy}. 

Let $\rho$ be as in item (c) of Theorem \ref{Thm:Expansion_Whom}. Note that $\rho \left( x', x_{3}, h+ |h \widetilde{E}_{h}(v_{h})| \right)$ converges to $0$ pointwise a.e. as $h \to 0$. Then, by Lemma \ref{lem:egorov}, there exists a sequence of measurable sets $D_{h} \subset \Omega$ such that 
\begin{enumerate}[(\alph*)]
\item[(a)]  $ \lim_{h \to 0}\left| \Omega \setminus D_{h} \right| =0 $.

\item[(b)] $\rho \left( x', x_{3},h+ |h \widetilde{E}_{h}(v_{h})| \right) \mathbf{1}_{D_{h}}$ converges uniformly to $0$.
\end{enumerate}

We will now split the integral in equation\eqref{eq:upboundint} into two parts: the integral over $D_{h}$ and the integral over $\Omega \setminus D_{h}$. We first deal with the integral over $\Omega \setminus D_{h}$. 

Note that by item $(W5)$ of Definition \ref{Def:StoredEnergyFct}, we have that, for some constant $C >0$,
\begin{equation}
\begin{split}
&\int_{\Omega \setminus D_{h}}\frac1{h^2} W^h_\ho \Big(x, \mathrm{Id} +h\widetilde{E}_{h}(v_{h}) \Big) \dd x \\
\leq & C \int_{\Omega \setminus D_{h}} \left| \widetilde{E}_{h}(v_{h}) \right|^{2} + 1 \dd x  \\
\leq &C \int_{\Omega \setminus D_{h}} \left| \widetilde{E}_{h}(v_{h}) - \iota(x_3\II_{\deform}+M^{*})+\sym(d^{*} \otimes e_{3}) \right|^{2} + \left| \iota(x_3\II_{\deform}+M^{*})+\sym(d^{*} \otimes e_{3}) \right|^{2} + 1 \dd x.
\end{split}
\end{equation}

The integral of the first term of the last line tends to $0$ in account of equation \eqref{eq:strongconvergence}. 
The integral of the second and third terms tend to $0$ because of Dominated Convergence Theorem. Also by Dominated Convergence Theorem, we have that 
\begin{equation}
\lim_{h \to 0} \left| \int_{\Omega \setminus D_{h}} Q_{\ho}(x',x_3,\iota(x_3\II_{\deform}+M^{*})+\sym(d^{*} \otimes e_{3}) -B_\ho) + R(B)\dd x \right|=0.
\end{equation}

Putting everything together, we have that 
\begin{equation}
\begin{split}
&\lim_{h \to 0} \Bigg| \int_{\Omega \setminus D_{h}}\frac1{h^2} W^h_\ho \Big(x, \mathrm{Id} +h\widetilde{E}_{h}(v_{h}) \Big) \dd x \\
&\ \ \ \ - \int_{\Omega \setminus D_{h}} Q_{\ho}(x',x_3,\iota(x_3\II_{\deform}+M^{*})+\sym(d^{*} \otimes e_{3}) -B_\ho) + R(B) \dd x \Bigg|\\
=&0.
\end{split}
\end{equation}

On the other hand, proceeding as in Subsection \ref{SS:lowerbound}, by equation \eqref{eq:strongconvergence}, item (c) of Theorem \ref{Thm:Expansion_Whom}, and by definition of $D_{h}$ we have that 
\begin{equation}
\begin{split}
&\lim_{h \to 0} \Bigg| \int_{D_{h}}\frac1{h^2} W^h_\ho \Big(x, \mathrm{Id} +h\widetilde{E}_{h}(v_{h}) \Big) \dd x \\
&\ \ \ \ \ \ \ \ - \int_{D_{h}} Q_{\ho}(x',x_3,\iota(x_3\II_{\deform}+M^{*})+\sym(d^{*} \otimes e_{3}) -B_\ho)+ R(B) \dd x \Bigg| \\ 
\leq &\lim_{h \to 0} \left| \int_{D_{h}}\frac1{h^2} W^h_\ho \Big(x, \mathrm{Id} +h\widetilde{E}_{h}(v_{h}) \Big) \dd x - \int_{D_{h}} Q_{\ho}(x',x_3,\widetilde{E}_{h}(v_{h})) + R(B)\dd x \right|  \\
&+\lim_{h \to 0} \Bigg| \int_{D_{h}} Q_{\ho}(x',x_3,\widetilde{E}_{h}(v_{h}))\dd x \\
&\ \ \ \ \ \ \ \ \ - \int_{D_{h}} Q_{\ho}(x',x_3,\iota(x_3\II_{\deform}+M^{*})+\sym(d^{*} \otimes e_{3}) -B_\ho)\dd x \Bigg|.
\end{split}
\end{equation}

Note that, by definition of $D_{h}$,
\begin{equation}
\begin{split}
&\lim_{h \to 0} \left| \int_{D_{h}}\frac1{h^2} W^h_\ho \Big(x, \mathrm{Id} +h\widetilde{E}_{h}(v_{h}) \Big) \dd x - \int_{D_{h}} Q_{\ho}(x',x_3,\widetilde{E}_{h}(v_{h})) + R(B)\dd x \right| \\
\leq &\lim_{h \to 0} \int_{\Omega} \rho \left( x', x_{3},h+ |h \widetilde{E}_{h}(v_{h})| \right) \left( 1 + \left| \widetilde{E}_{h}(v_{h})\right| \right)  \mathbf{1}_{D_{h}} \dd x \\
\leq &\lim_{h \to 0}  \sup_{x \in D_{h}} \rho \left( x', x_{3},h+ |h \widetilde{E}_{h}(v_{h})| \right) \int_{\Omega} \left( 1 + \left| \widetilde{E}_{h}(v_{h})\right| \right) \dd x\\
= &0.
\end{split}
\end{equation}

On the other hand, item (b) of the properties of $v_{h}$ implies that
\begin{equation}
\begin{split}
&\lim_{h \to 0} \Bigg| \int_{D_{h}} Q_{\ho}(x',x_3,\widetilde{E}_{h}(v_{h}))\dd x \\
&\ \ \ \ \ \ \ \ \ - \int_{D_{h}} Q_{\ho}(x',x_3,\iota(x_3\II_{\deform}+M^{*})+\sym(d^{*} \otimes e_{3}) -B_\ho)\dd x \Bigg|\\
=&0. 
\end{split}
\end{equation}

From this, we can conclude.
\end{proof}

\subsection{Characterization of the energy}
\label{sect:alg}

In this section we show that indeed $\widetilde{\mathcal{E}}^{\infty}_{\hom} = \mathcal{E}^{\infty}_{\hom}$, hence completing the proof of Theorem \ref{T1}. This is a consequence of the algebraic machinery of Section~\ref{sect:Qeff}.

\begin{lemma}
\label{L:rewriting}
Let $\mathcal{E}^{\infty}_{\hom}$ be given by equation \eqref{eq:limitgamma} with $\gamma=\infty$, and $Q^\infty_{\eff}, B^\infty_{\eff}, $ and $\mathcal I_{\rm res}^\infty$ given by Definitions \ref{def:Qhom}, \ref{D:eff} and \ref{D:Ires}, respectively. Let $\widetilde{\mathcal{E}}^{\infty}_{\hom}$ be given by equation \eqref{eq:Etilde} or, equivalently, by equation \eqref{eq:Etildeequiv}. Then 
\begin{equation}
\widetilde{\mathcal{E}}^{\infty}_{\hom} = \mathcal{E}^{\infty}_{\hom}. 
\end{equation}
\end{lemma}

\begin{proof}
Fix $x' \in S$ and note that, for any $\chi \in \mathbf{H}^{\infty}_{\rm rel}$ and $G \in \mathbb{R}^{2 \times 2}_{\sym}$ there holds
\begin{equation}
\begin{split}
&\int_{\square} Q \left( x', x_{3}, y, \iota \left( x_3 G \right) + \chi - B(x', y, x_{3}) \right) \dd (x_3,y)\\
=& \left\| \iota \left( x_3 G \right) + \chi + P^{\infty}\sym(B)(x', \cdot) \right\|^{2}_{x'} + \left\| B(x', \cdot)- P^{\infty}\sym(B)(x', \cdot) \right\|^{2}_{x'}. 
\end{split}
\end{equation}

Furthermore, by Definitions \ref{def:Qhom} and \ref{D:eff} there holds
\begin{equation}
\begin{split}
\inf_{\chi \in \mathbf{H}^{\infty}_{\rm rel}} \left\| \iota \left( x_3 G \right) + \chi + P^{\infty}\sym(B)(x', \cdot) \right\|^{2}_{x'} &=  \left\| P^{\infty, \perp}_{\rm rel}\left(  \iota \left( x_{3} \left(G - B^{\infty}_{\eff}(x') \right) \right) \right) \right\|^{2}_{x'} \\
&= Q^{\infty}_{\eff}(x', G - B^{\infty}_{\eff}(x')). 
\end{split}
\end{equation}

Integrating over $x'$ and using Definition \ref{D:Ires}, we may conclude. 
\end{proof}
As a corollary of Theorem \ref{T1'} and Lemma \ref{L:rewriting}, we obtain Theorem \ref{T1}.

\section{Proof of Theorems \ref{T:quant}, \ref{T:quant2}, and Proposition \ref{prop:example}}
\label{sect:quant}

In this section, we will prove the results about a rate of convergence to the limit in Theorem \ref{T:gammainfty}, that is we prove Theorems \ref{T:quant}, Theorem \ref{T:quant2}, and Proposition \ref{prop:example}. For the proof, we will need two lemmas. We state the first one. 

\begin{lemma}
\label{lem:regularity}
Let $Q$ and $B$ be an admissible quadratic form and prestrain, respectively. 
Let $H\in \HH$, fix $x' \in S$ and let $(M,\varphi, d)$ denote the corrector for $\gamma=\infty$ associated with $H$, see Lemma \ref{L:corrector}. 
Assume \eqref{eq:alpharegularity} for some $\alpha\in(0,1]$ and 
\begin{equation*}
	H\in\begin{cases}
		C^{\alpha} (I, L^{2}(\mathcal{Y}, \R^{3}))&\text{if }\alpha\in(0,1),\\
		W^{1, \infty} (I, L^{2}(\mathcal{Y}, \R^{3}))&\text{if }\alpha=1.
	\end{cases}
\end{equation*}
Then, for $\alpha \in (0,1)$
\begin{equation}
\label{eq:holderreg}
\sym \nabla' \varphi \in C^{\alpha} (I, L^{2}(\mathcal{Y}, \R^{2\times 2})), \qquad d \in C^{\alpha} (I, L^{2}(\mathcal{Y}, \R^{3})),
\end{equation}
and for $\alpha =1$, 
\begin{equation}
\label{eq:lipschitzreg}
\sym \nabla' \varphi \in W^{1, \infty}(I, L^{2}(\mathcal{Y}, \R^{2\times 2})), \qquad d \in W^{1, \infty}(I, L^{2}(\mathcal{Y}, \R^{3})).
\end{equation}
\end{lemma}

\begin{proof}
We begin by noting that, for every $x_{3} \in I$, $\varphi(x_{3}, \cdot), d(x_{3}, \cdot)$ satisfies
  \begin{equation}
  \label{eq:noshift}
    \int_{\mathcal{Y}}\mathbb L(x', x_{3}, y)\left[H(x_{3}, y)+\iota(M)+\sym(\nabla'\varphi(x_{3}, y) | d(x_{3}))\right]:\big(\sym(\nabla'\varphi'| d')\big)\dd y=0
  \end{equation}
  for all $\varphi' \in W^{1,2}(\mathcal{Y}, \R^{3})$, $d' \in \R^{3}$. Similarly, for any $-x_{3} + \frac{1}{2}< \delta < \frac{1}{2}-x_{3}$, $\varphi(x_{3}+ \delta, \cdot)$ and $d(x_{3} + \delta, \cdot)$ satisfy
  \begin{equation}
  \label{eq:shift}
    \int_{\mathcal{Y}}\mathbb L(x', x_{3}+\delta,y)
    \begin{aligned}[t]
        &\left(\sym(\nabla'\varphi'| d')\right) \\
        &: \left[H(x_{3} + \delta, y)+\iota(M)+\sym(\nabla'\varphi(x_{3} + \delta, y)| d(x_{3} + \delta))\right] \dd y=0
    \end{aligned}
  \end{equation}
 for all $\varphi' \in W^{1,2}(\mathcal{Y}, \R^{3})$, $d' \in \R^{3}$. Given $x_{3} \in I$, and $\delta < 1-x_{3}$, we define the finite difference quotients 
\begin{equation}
\varphi_{1}^{\delta} (x_{3}, \cdot ) \colonequals \frac{\varphi(x_{3} + \delta, \cdot ) - \varphi( x_{3}, \cdot  )}{\delta^{\alpha}} \in W^{1,2}(\mathcal{Y}, \R^{3}) \qquad d_{1}^{\delta} \colonequals \frac{d(x_{3} + \delta) - d( x_{3} )}{\delta^{\alpha}} \in \R^{3}.
\end{equation}  

Note that equations \eqref{eq:noshift} and \eqref{eq:shift} imply that the difference quotients satisfy the equation
  \begin{equation}  
  \label{eq:quot}
  \begin{split}
    &\int_{\mathcal{Y}} \Bigg( \partial_{3}^{\delta}\mathbb L(x', x_{3}, y)\left[H(x_{3}, y) + \iota(M)+\sym(\nabla'\varphi(x_{3}, y)| d(x_{3}))\right]\\
    & \qquad +  (\mathbb L(x', x_{3}+\delta, y)\left[\partial_{3}^{\delta} H(x_{3}, y)+\sym(\nabla'\varphi_{1}^{\delta}(x_{3}, y)| d_{1}^{\delta}(x_{3}))\right] \Bigg) \\
    & \qquad : \big(\sym(\nabla'\varphi'| d')\big)\dd y=0
  \end{split}   
  \end{equation}
 for all $\varphi' \in W^{1,2}(\mathcal{Y}, \R^{3})$, $d' \in \R^{3}$, where $\partial_{3}^{\delta}\mathbb L(x', x_{3}, y) \colonequals \frac{\mathbb L(x', x_{3}+\delta, y) - \mathbb L(x', x_{3}, y)}{\delta^{\alpha}}$, and $\partial_{3}^{\delta}H(x_{3}, y)\colonequals \frac{ H(x_{3}+\delta, y) -  H(x_{3}, y)}{\delta^{\alpha}}$.  
 
We may now proceed as in step 2 of the proof of Lemma \ref{L:corrector} to obtain the desired bound: choosing $\sym(\nabla'\varphi'(y)| d')= \sym(\nabla'\varphi_{1}^{\delta}(x_{3}, y)| d_{1}^{\delta}(x_{3}))$ as a test function we obtain
  \begin{equation}  
  \label{eq:quot2}
  \begin{split}
    &\int_{\mathcal{Y}} \Bigg( \partial_{3}^{\delta}\mathbb L(x', x_{3}, y)\big(H(x_{3}, y) + \iota(M)+\sym(\nabla'\varphi(x_{3}, y)| d(x_{3}))\big)\\
    & \qquad +  (\mathbb L(x', x_{3}+\delta, y)\big(\partial_{3}^{\delta} H(x_{3}, y)\big) \Bigg) : \big(\sym(\nabla'\varphi_{1}^{\delta}(x_{3}, y)| d_{1}^{\delta}(x_{3}))\big)\dd y\\
    &=\int_{\mathcal{Y}} (\mathbb L(x', x_{3}+\delta, y)\sym(\nabla'\varphi_{1}^{\delta}(x_{3}, y)| d_{1}^{\delta}(x_{3})) : \sym(\nabla'\varphi_{1}^{\delta}(x_{3}, y)| d_{1}^{\delta}(x_{3}))\dd y.
  \end{split}   
  \end{equation}
Note that  
\begin{equation}
\begin{split}
&C \int_{\mathcal{Y}} (\mathbb L(x', x_{3}+\delta, y)\sym(\nabla'\varphi_{1}^{\delta}(x_{3}, y)| d_{1}^{\delta}(x_{3})) : \sym(\nabla'\varphi_{1}^{\delta}(x_{3}, y)| d_{1}^{\delta}(x_{3}))\dd y\\
\geq & \|\nabla'\varphi_{1}^{\delta}(x_{3}, \cdot)\|_{L^{2}(\mathcal{Y})}^{2} +  \left|d_{1}^{\delta}(x_{3})\right|^{2}. 
\end{split}
\end{equation}
 Using Cauchy-Schwartz on the LHS of equation \eqref{eq:quot2} and we have the bounds
\begin{equation}
\begin{split}
&\|\nabla'\varphi_{1}^{\delta}(x_{3}, \cdot)\|_{L^{2}(\mathcal{Y})}^{2} +  \left|d_{1}^{\delta}(x_{3})\right|^{2} \\
\leq & C\left(\|\partial_{3}^{\delta}H(x_{3}, \cdot)\|_{L^{2}}^{2}+\sup_{y \in \mathcal{Y}; x_{3}^{1}, x_{3}^{2} \in I } \frac{\left| \mathbb L(x', x_{3}^{1}, y) - \mathbb L(x', x_{3}^{2}, y)\right|}{|x_{3}^{2} - x_{3}^{1}|^{\alpha}}\|H(x_{3}, \cdot)\|_{L^{2}}^{2}\right),
\end{split}
\end{equation}
where $C$ depends only on $\alpha, \beta$. Equations \eqref{eq:holderreg} and \eqref{eq:lipschitzreg} follow. 

\end{proof}

Before proceeding to the proof of Theorem \ref{T:quant}, we will need a lemma concerning regularization by convolution.

\begin{lemma}
\label{lem:convolution}
Let $\mu: \R \to \R^{+}$ be the standard mollifier (with support in $I$) and define, for $l\leq 1$, the mollifier $\mu_{l}$ as 
\begin{equation}
\mu_{l}(x_{3}) \colonequals \frac{1}{l} \mu \left( \frac{x_{3}}{l} \right).
\end{equation}
Then there exists a constant $C$, depending only on $\mu$, such that for any $\varphi \in L^{2}(I, W^{1,2}(\mathcal{Y}, \R^{3}))$,\footnote{We are committing an abuse of notation by writing $\varphi \ast \mu_{l}$, since $\varphi(x_{3}, \cdot)$ is not defined for $x_{3}$ outside of $I$. This can be fixed by extending $\varphi$ by odd reflection to the interval $\left(-\frac{3}{2}, \frac{3}{2} \right)$.} 
\begin{equation}
\label{eq:firstconvolution}
\left\| \nabla' \varphi - \nabla' \varphi \ast \mu_{l} \right\|_{L^{2}(\square)} \leq C l^{\alpha} \sup_{x_{3}, x_{3} +z \in I} \left\| \frac{ \nabla' \varphi (x_{3}, y) - \nabla' \varphi(x_{3} + z, y)}{|z|^{\alpha}} \right\|_{L^{2}(\mathcal{Y})}
\end{equation}
and 
\begin{equation}
\label{eq:secondconvolution}
\left\| \partial_{3} \varphi \ast \mu_{l} \right\|_{L^{2}(\square)} \leq C l^{\alpha-1} \sup_{x_{3}, x_{3} +z \in I} \left\| \frac{ \varphi (x_{3}, y) - \varphi(x_{3} + z, y)}{|z|^{\alpha}} \right\|_{L^{2}(\mathcal{Y})}.
\end{equation}
\end{lemma}

\begin{proof}
\step 1 Proof of equation \eqref{eq:firstconvolution}.\\

\begin{equation}
\begin{split}
\left\| \nabla' \varphi - \nabla' \varphi \ast \mu_{l} \right\|_{L^{2}}^{2} &= \int_{\Box}  \left| \int_{-\frac{l}{2}}^{\frac{l}{2}} |z|^{\alpha}  \frac{\nabla' \varphi (x_{3}, y) - \nabla' \varphi(x_{3} + z, y) }{|z|^{\alpha}} \mu_{l}(z)\, \mathrm d z \right|^{2} \dd (x_3,y)\\
&\leq \int_{\Box}   \int_{-\frac{l}{2}}^{\frac{l}{2}} \left|  |z|^{\alpha} \frac{\nabla' \varphi (x_{3}, y) - \nabla' \varphi(x_{3} + z, y) }{|z|^{\alpha}} \right|^{2}\mu_{l}(z) \, \mathrm d z \ \dd (x_3,y)\\
&\leq l^{2 \alpha}  \int_{-\frac{l}{2}}^{\frac{l}{2}} \sup_{x_{3} \in I} \left\| \frac{ \nabla' \varphi (x_{3}, y) - \nabla' \varphi(x_{3} + z, y)}{|z|^{\alpha}} \right\|_{L^{2}(\mathcal{Y})}^{2} \mu_{l}(z) \, \mathrm d z \\
&\leq C l^{2\alpha}  \sup_{x_{3}, x_{3} +z \in I} \left\| \frac{ \nabla' \varphi (x_{3}, y) - \nabla' \varphi(x_{3} + z, y)}{|z|^{\alpha}} \right\|_{L^{2}(\mathcal{Y})}^{2}. 
\end{split}
\end{equation}
We have used Jensen's inequality to pass from line 1 to line 2.

\step 2 Proof of equation \eqref{eq:secondconvolution}.\\

\begin{equation}
\begin{split}
\left\| \partial_{3} \varphi \ast \mu_{l} \right\|_{L^{2}}^{2} &= \int_{\Box}  \left| \int_{-\frac{l}{2}}^{\frac{l}{2}} \varphi(x_{3} + z, y)\frac{\mu'\left(\frac{z}{l} \right)}{l^{2}} \, \mathrm d z \right|^{2} \dd (x_3,y)\\   
&= l^{-2} \int_{\Box}  \left| \int_{-\frac{l}{2}}^{\frac{l}{2}} \frac{\varphi(x_{3} + z, y) - \varphi(x_{3}, y)}{|z|^{\alpha}} |z|^{\alpha}\frac{\mu'\left(\frac{z}{l} \right)}{l} \, \mathrm d z \right|^{2} \dd (x_3,y)\\ 
&\leq  C l^{2(\alpha-1)} \int_{\Box}   \int_{-\frac{l}{2}}^{\frac{l}{2}} \left|\frac{\varphi(x_{3} + z, y) - \varphi(x_{3}, y)}{|z|^{\alpha}} \right|^{2} \left |\frac{\mu'\left(\frac{z}{l} \right)}{l} \right|  \, \mathrm d z \ \dd (x_3,y)\\ 
&\leq C l^{2(\alpha-1)} \int_{-\frac{l}{2}}^{\frac{l}{2}} \sup_{z>0} \left\| \frac{  \varphi (x_{3} + z, y) -  \varphi(x_{3} , y)}{|z|^{\alpha}} \right\|_{L^{2}}^{2}\left |\frac{\mu'\left(\frac{z}{l} \right)}{l} \right|  \, \mathrm d z \\ 
&\leq C l^{2(\alpha - 1)} \sup_{x_{3}, x_{3} +z \in I} \left\| \frac{  \varphi (x_{3} + z, y) - \varphi(x_{3}, y)}{|z|^{\alpha}} \right\|_{L^{2}(\mathcal{Y})}^{2}. 
\end{split}
\end{equation}
In order to pass from line 1 to line 2, we have used that $\mu'$ has integral $0$. In order to pass from line 2 to line 3, we have used Jensen's inequality for the probability measure $\left |\frac{\mu'\left(\frac{z}{l} \right)}{l \int |\mu'|} \right| $.
\end{proof}

Having proved Lemmas \ref{lem:regularity} and \ref{lem:convolution} , we now give the proof of Theorem \ref{T:quant}.  

\begin{proof}[Proof of Theorem \ref{T:quant}]

\step 1 Let $(M_{\infty}, \varphi_{\infty}, d_{\infty})$ denote the corrector at $\infty$ associated to $\iota(x_3 G)$ in the sense of Lemma \ref{L:corrector}. We claim that, if $Q$ is orthotropic, then $\varphi_{\infty} \cdot e_{3}=0$ and $d_{\infty}=0$.\\

To prove this claim, note that, using the hypothesis that $Q$ is orthotropic, we have that 
\begin{equation}
\begin{split}
Q^{\infty}_{\eff} (x', G) &= \int_{\Box }Q\left(x', x_3,y,\iota(x_3G + M_{\infty})+ \left( \nabla'\varphi_{\infty} | d_{\infty} \right) \right)\dd (x_3,y) \\ 
&= \int_{\Box }Q\left(x', x_3,y,\iota(x_3G+ M_{\infty})+ \iota \left( \nabla'\varphi_{\infty} \right)_{2 \times 2} \right) \\
&\ \ \ \ \ \ + Q\left(x', x_3,y, e_{3}\otimes (\partial_{1} \varphi_{\infty}^{3}, \partial_{2} \varphi_{\infty}^{3}, 0) + d_{\infty}\otimes e_{3} \right)\dd (x_3,y). 
\end{split}
\end{equation}

Hence, setting $\widetilde{\varphi}_{\infty} := (\varphi_{\infty}^{1}, \varphi_{\infty}^{2}, 0)$, we get that 
\begin{equation}
\int_{\Box }Q\left(x', x_3,y,\iota(x_3G)+ \left( \nabla'\widetilde{\varphi}_{\infty} | 0 \right) \right)\dd (x_3,y) \leq Q^{\infty}_{\eff} (x',G).  
\end{equation}

By uniqueness of the corrector, this implies that $\widetilde{\varphi}_{\infty} = {\varphi}_{\infty}$ and $d_{\infty} = 0$. 

\step 2 Proof of item (a).\\

\medskip\noindent\textit{Substep 2.1 }Proof of upper bound for $Q^{\gamma}_{\eff} (x', G) $. For this substep, we assume either that $Q$ is even in $x_{3}$ or is orthotropic. \\ 

Let $\mu_{l}$ be as in Lemma \ref{lem:convolution}, with $l=\gamma^{-p}$, for $p>0$ to be determined later. 

Define
\begin{equation}
D_{\infty}(x_{3}) \colonequals \int_{0}^{x_{3}} d_{\infty}(s) \dd s.
\end{equation}  
We use $\varphi_{\infty} \ast \mu_{l} + \gamma D_{\infty} $ as a test function in the definition of $Q^{\gamma}_{\eff} (x', G)$ and get:
\begin{equation}
\label{eq:comparison}
\begin{split}
Q^{\gamma}_{\eff} (x', G) &\leq \int_{\Box }Q\left(x', x_3,y,\iota(x_3G+ M_{\infty})+ \nabla_{\gamma} \left( \varphi_{\infty}\ast \mu_{l} + \gamma D_{\infty} \right) \right)\dd (x_3,y) \\
&= \int_{\Box }Q\Bigg(x', x_3,y,\iota(x_3G +\iota(M_{\infty}) +\left( \nabla'\varphi_{\infty} | d_{\infty} \right)+ (\nabla'\varphi_{\infty}\ast \mu_{l}-\nabla'\varphi_{\infty}|0)\\
&\ \ \ \ \ \ \ \ \ \ \ \ + \frac{1}{\gamma} \left( 0 | 0 | \partial_{3} \varphi_{\infty}\ast \mu_{l} \right) \Bigg)\dd (x_3,y) \\
&= Q^{\infty}_{\eff} (x', G) + \int_{\Box }\frac{1}{\gamma^{2}}Q\left(x', x_3,y,  \left( 0 | 0 | \partial_{3 }\varphi_{\infty}\ast \mu_{l} \right) \right)\\
&\ \ \ \ \ \ \ \ \ \ \ \ + Q\left(x', x_3,y, (\nabla'\varphi_{\infty}\ast \mu_{l}-\nabla'\varphi_{\infty}|0) \right)\dd (x_3,y).
\end{split}
\end{equation}

In equation \eqref{eq:comparison}, we have used that, if either $Q$ is even in $x_{3}$ or orthotropic, then\footnote{For the remainder of this section, we omit the dependence of $\mathbb L$ on $(x', x_{3}, y)$ for ease of notation.} 
\begin{equation}
\label{eq:innerprod}
\begin{split}
\int_{\Box }\mathbb L \left(\iota(x_3G) +\iota(M_{\infty}) + \left( \nabla'\varphi_{\infty} | d_{\infty} \right) : \left( 0 | 0 | \partial_{3}\varphi_{\infty} \ast \mu_{l}  \right) \right)\dd (x_3,y) &=0 \\
\int_{\Box }\mathbb L \left((\nabla'\varphi_{\infty}\ast \mu_{l}-\nabla'\varphi_{\infty}|0) : \left( 0 | 0 | \partial_{3}\varphi_{\infty} \ast \mu_{l}  \right) \right)\dd (x_3,y) &=0, 
\end{split}
\end{equation}
while the Euler-Lagrange equation for the corrector (see Lemma \ref{L:corrector}) implies that
\begin{equation}
\int_{\Box }\mathbb L \left(\iota(x_3G) +\iota(M_{\infty}) + \left( \nabla'\varphi_{\infty} | d_{\infty} \right) : (\nabla'\varphi_{\infty}\ast \mu_{l}-\nabla'\varphi_{\infty}|0) \right)\dd (x_3,y) =0.
\end{equation}

From equation \eqref{eq:comparison} and Lemma \ref{lem:convolution}, we get that 
\begin{equation}
Q^{\gamma}_{\eff} (x', G) \leq Q^{\infty}_{\eff} (x', G) + c|G|^{2} \left(\frac{l^{2(\alpha - 1)}}{\gamma^{2}} + l^{2 \alpha} \right).
\end{equation}
The optimal $l$ is given by $l=\gamma^{-1}$. From this, we may conclude. 

\medskip\noindent\textit{Substep 2.2 }Proof of upper bound for $Q^{\infty}_{\eff} (x', G) $. For the remainder of the proof, we return to assuming that $Q$ is orthotropic. \\ 

Let $(M_{\gamma},\varphi_{\gamma})$ denote the corrector associated to $\iota(x_3 G)$ for $\gamma \in (0, \infty)$. Denote $\varphi_{\gamma} = (\varphi_{\gamma}', \varphi_{\gamma}^{3})$. Consider as a test function in the definition of $Q^{\infty}_{\eff} (x', G)$, $\varphi=\varphi_{\gamma}'$ and $d=0$. Then, using the hypothesis that $Q$ is orthotropic, 
\begin{equation}
\begin{split}
Q^{\infty}_{\eff} (x', G) &\leq \int_{\Box }Q\left(x', x_3,y,\iota(x_3G + M_{\gamma})+ \left( \nabla'\varphi_{\gamma}' | 0 \right) \right)\dd (x_3,y) \\ 
&\leq  \int_{\Box }Q\left(x', x_3,y,\iota(x_3G + M_{\gamma})+ \left( \nabla'\varphi_{\gamma}' | 0 \right) \right)\dd (x_3,y)  \\
&\ \ \ \ \ \ + Q\left(x', x_3,y, e_{3}\otimes \nabla'\varphi_{\gamma}^{3} + \frac{1}{\gamma}\partial_{3}\varphi_{\gamma}\otimes e_{3} \right)\dd (x_3,y)\\
&= \int_{\Box }Q\left(x', x_3,y,\iota(x_3G)+ \iota(M_{\gamma}) + \nabla_{\gamma} \varphi_{\gamma} \right)\dd (x_3,y) \\
&=Q^{\gamma}_{\eff} (x', G). 
\end{split}
\end{equation}

\step 3 We claim that, if $Q$ is orthotropic, then there exists a constant $C$ (depending on $G$) such that
\begin{equation}
\left\| \sym \left( \nabla_{\gamma} \varphi_{\gamma} \right) + \iota(M_{\gamma}) - \sym \left( \nabla'\varphi_{\infty} | d_{\infty} \right) + \iota(M_{\infty}) \right\|_{L^{2}} \leq \frac{C}{{\gamma^{\alpha}}}.
\end{equation} \\

To prove this claim, note that 
\begin{equation}
\begin{split}
Q^{\gamma}_{\eff} (x', G) &= \int_{\Box }Q\left(x', x_3,y,\iota(x_3G)+ \iota(M_{\gamma}) + \nabla_{\gamma} \varphi_{\gamma} \right)\dd (x_3,y) \\
&= \int_{\Box }Q\Bigg(x',  x_3,y,\iota(x_3G)+ \iota(M_{\infty})+ \left( \nabla'\varphi_{\infty} | 0 \right) + \iota(M_{\gamma}) - \nabla_{\gamma} \varphi_{\gamma} \\
&\ \ \ \ \ - \iota(M_{\infty})- \left( \nabla'\varphi_{\infty} | 0 \right) \Bigg)\dd (x_3,y).
\end{split}
\end{equation}

Next, we note that the EL equation (equation \eqref{eq:corrector_equation}) for $(\varphi_{\infty}, d_{\infty})$ implies that
\begin{equation}
\int_{\Box }\mathbb L \left(\iota(x_3G + \iota(M_{\infty})+\left( \nabla'\varphi_{\infty} | 0 \right) \right): \left( \iota(M_{\gamma} -M_{\infty} ) + \left(  \nabla'\varphi_{\gamma}-  \nabla'\varphi_{\infty} | 0 \right) \right)\dd (x_3,y) =0. 
\end{equation}

On the other hand, the hypothesis that $Q$ is orthotropic implies that
\begin{equation}
\mathbb L \left(\iota(x_3G +\iota(M_{\infty})+ \left( \nabla'\varphi_{\infty} | 0 \right)\right) : e_{3}\otimes \partial_{3}\varphi_{\gamma} =0.
\end{equation}

Therefore 
\begin{equation}
\label{eq:lasteq}
\begin{split}
Q^{\gamma}_{\eff} (x', G) &= Q^{\infty}_{\eff} (x', G) +  \int_{\Box }Q\left(x', x_3,y,\iota(M_{\gamma}) - \nabla_{\gamma} \varphi_{\gamma} - \iota(M_{\infty}) - \left( \nabla'\varphi_{\infty} | 0 \right) \right)\dd (x_3,y)\\
&\geq Q^{\infty}_{\eff} (x', G) + c \left\| \sym \left( \nabla_{\gamma} \varphi_{\gamma} \right) + \iota(M_{\gamma}) - \sym \left( \nabla'\varphi_{\infty} | d_{\infty} \right) + \iota(M_{\infty}) \right\|_{L^{2}}^{2}. 
\end{split}
\end{equation}

From equation \eqref{eq:lasteq}, and item (a), we may conclude. 

\step 4 Proof of item (b).\\

Let $(M_{\infty}^{i},\varphi_{\infty}^{i}, d_{\infty}^{i})$ denote the corrector at $\infty$ associated to $\iota \left( x_3 G_i \right)$ in the sense of Lemma \ref{L:corrector}, and let $(M_{\gamma}^{i},\varphi_{\gamma}^{i})$ denote the corrector at $\gamma$ associated to $\iota \left( x_3 G_i \right)$ (see step 2 of the proof of Theorem \ref{T:gammainfty}). Using Proposition \ref{P:1} and Cauchy-Schwartz inequality, we have that
\begin{equation}
\begin{split}
&\left| B^{\gamma}_{\eff}(x') - B^{\infty}_{\eff}(x') \right|\\
 \leq &C \sum_{i=1}^{3}\left| \int_{\Box }\mathbb L\big( \sym \left( \nabla_{\gamma} \varphi_{\gamma}^{i} \right) + \iota(M^{i}_{\gamma}) - \sym \left( \nabla'\varphi^{i}_{\infty} | d^{i}_{\infty} \right) + \iota(M^{i}_{\infty}) \big):B(x', x_{3}, y)\dd (x_3,y)\right|\\
\leq & C \sum_{i=1}^{3} \left \|  \sym \left( \nabla_{\gamma} \varphi^{i}_{\gamma} \right) + \iota(M^{i}_{\gamma}) - \sym \left( \nabla'\varphi^{i}_{\infty} | d^{i}_{\infty} \right) + \iota(M^{i}_{\infty}) \right\|_{L^{2}} \left\| B(x', \cdot) \right\|_{L^{2}(\Box)}.
\end{split}
\end{equation}

Using item (a), we may conclude. 

\step 5 Proof of item (c).\\

Using Lemma \ref{L:rewriting} and choosing $\deform(x') = x'$ so that $\II_{\deform} = 0$, we have that 
\begin{equation}
\int_{S} Q_{\rm ext}^{\infty} (x', B(x', \cdot)) \dd x'=  \mathcal I_{\rm res}^\infty + \int_{S} Q^{\infty}_{\eff} \left( x',  B_{\eff}^{\infty} \right) \dd x'. 
\end{equation}
Similarly, for any $\gamma \in (0, \infty)$ there holds
\begin{equation}
\int_{S} Q_{\rm ext}^{\gamma} (x', B(x', \cdot)) \dd x'=  \mathcal I_{\rm res}^\gamma + \int_{S} Q^{\gamma}_{\eff} \left( x',  B_{\eff}^{\gamma} \right) \dd x'. 
\end{equation}
Hence,
\begin{equation}
\begin{split}
&\left| \mathcal I_{\rm res}^\gamma - \mathcal I_{\rm res}^\infty  \right|\\
\leq &  \int_{S} \left| Q_{\rm ext}^{\infty} (x', B(x', \cdot)) - Q_{\rm ext}^{\infty} (x', B(x', \cdot)) \right|  \dd x'+ \int_{S} \left| Q^{\infty}_{\eff} \left( x',  B_{\eff}^{\infty} \right) - Q^{\gamma}_{\eff} \left( x',  B_{\eff}^{\gamma} \right) \right| \dd x'.
\end{split}
\end{equation}
 
Furthermore
\begin{equation}
\begin{split}
&\int_{S} \left| Q^{\infty}_{\eff} \left( x',  B_{\eff}^{\infty} \right) -  Q^{\gamma}_{\eff} \left( x',  B_{\eff}^{\gamma} \right) \right| \dd x'\\
\leq &\int_{S} \left| Q^{\infty}_{\eff} \left( x',  B_{\eff}^{\infty} \right) -  Q^{\infty}_{\eff} \left( x',  B_{\eff}^{\gamma} \right) \right|  \dd x' + \left| Q^{\infty}_{\eff} \left( x',  B_{\eff}^{\gamma} \right) -  Q^{\gamma}_{\eff} \left( x',  B_{\eff}^{\gamma} \right) \right| \dd x'.
\end{split}
\end{equation} 

By item (a), we have that 
\begin{equation}
\int_{S} \left| Q^{\infty}_{\eff} \left( x',  B_{\eff}^{\gamma} \right) -  Q^{\gamma}_{\eff} \left( x',  B_{\eff}^{\gamma} \right) \right| \dd x' \leq \frac{C}{\gamma^{2 \alpha}}.
\end{equation}

On the other hand, using polar decomposition, Cauchy-Schwartz inequality, and item (c), we find that 
\begin{equation}
\begin{split}
\int_{S} \left| Q^{\infty}_{\eff} \left( x',  B_{\eff}^{\infty} \right) -  Q^{\infty}_{\eff} \left( x',  B_{\eff}^{\gamma} \right) \right| \dd x' &= \int_{S} \left| \mathbb L^{\infty}_{\eff} \left( x',  B_{\eff}^{\infty} + B_{\eff}^{\gamma} \right): \left(  B_{\eff}^{\infty} - B_{\eff}^{\gamma} \right)  \right| \dd x' \\
&\leq C \int_{S} \left|    B_{\eff}^{\infty} - B_{\eff}^{\gamma} \right| \dd x' \\
&\leq \frac{C}{\gamma},
\end{split}
\end{equation} 
where
  \begin{equation}
     \mathbb L^{\infty}_{\eff} F:G\colonequals\frac12\big(Q^{\infty}_{\eff}(F+G)-Q^{\infty}_{\eff}(F)-Q^{\infty}_{\eff}(G)\big),
  \end{equation}

In order to find a bound for $\left|  Q_{\rm ext}^{\gamma} (x', B(x', \cdot)) - Q_{\rm ext}^{\infty} (x', B(x', \cdot)) \right|$, we proceed as in step 1: let $(M^{B}_{\infty}, \varphi^{B}_{\infty},d^{B}_{\infty})$ denote the corrector associated to $B(x', \cdot)$ at $x'$. Define
\begin{equation}
D^{B}_{\infty}(x_{3}) \colonequals \int_{0}^{x_{3}} d^{B}_{\infty}(s) \dd s.
\end{equation} 
 
Let $\mu_{l}$ be as in Lemma \ref{lem:convolution} with $l = \gamma^{-p}$, and use $\varphi^{B}_{\infty} \ast \mu_{l} + \gamma D^{B}_{\infty} $ as a test function in the definition of $Q_{\rm ext}^{\gamma} (x', B(x', \cdot))$. The result is that
\begin{equation}
\label{eq:ineq3}
\begin{split}
Q_{\rm ext}^{\gamma} (x', B(x', \cdot)) &\leq \int_{\Box }Q\left(x', x_3,y,B(x', x_{3}, y)+ \iota(M^{B}_{\infty})+ \nabla_{\gamma} \left( \varphi^{B}_{\infty}\ast \mu_{l} + \gamma D^{B}_{\infty} \right) \right)\dd (x_3,y) \\
&= \int_{\Box }Q\Bigg( x', x_3,y,B(x', x_{3}, y)+\iota(M^{B}_{\infty}) +\left( \nabla'\varphi^{B}_{\infty} | d^{B}_{\infty} \right)+ (\nabla'\varphi^{B}_{\infty}\ast \mu_{l}-\nabla'\varphi^{B}_{\infty}|0)\\
&\ \ \ \ \ + \frac{1}{\gamma} \left( 0 | 0 | \partial_{3} \varphi^{B}_{\infty}\ast \mu_{l} \right) \Bigg)\dd (x_3,y) \\
&= Q_{\rm ext}^{\infty} (x', B(x', \cdot)) + C \left( \frac{l^{\alpha-1}}{\gamma} + l^{\alpha} \right).
\end{split}
\end{equation}

In equation \eqref{eq:ineq3}, we have used that, by Cauchy-Schwartz and Lemma \ref{lem:convolution},
\begin{equation}
\begin{split}
\left| \int_{\Box }\mathbb L \left(B(x', x_{3}, y)+\iota(M^{B}_{\infty}) +\left( \nabla'\varphi^{B}_{\infty} | d^{B}_{\infty} \right) \right): \left( 0 | 0 | \partial_{3}\varphi^{B}_{\infty} \ast \mu_{l} \right) \dd (x_3,y) \right| & \leq C {l^{\alpha-1}} \\
\left| \int_{\Box }\mathbb L \left(\nabla'\varphi^{B}_{\infty}\ast \mu_{l}-\nabla'\varphi^{B}_{\infty}|0  \right): \left( 0 | 0 | \partial_{3}\varphi^{B}_{\infty}\ast \mu_{l} \right)\dd (x_3,y) \right| &\leq C {l^{2 \alpha-1}},\\
\left| \int_{\Box }\mathbb L \left(B(x', x_{3}, y)+\iota(M^{B}_{\infty}) +\left( \nabla'\varphi^{B}_{\infty} | d^{B}_{\infty} \right)  \right): (\nabla'\varphi^{B}_{\infty}\ast \mu_{l}-\nabla'\varphi^{B}_{\infty}|0)\dd (x_3,y) \right| &\leq C l^{\alpha}.
\end{split}
\end{equation}

Choosing $l=\gamma^{-1}$ in equation \eqref{eq:ineq3} we find that 
\begin{equation}
\left|  Q_{\rm ext}^{\gamma} (x', B(x', \cdot)) - Q_{\rm ext}^{\infty} (x', B(x', \cdot)) \right| \leq  \frac{C}{\gamma^{\alpha}}. 
\end{equation}
It follows that
\begin{equation}
\left| \mathcal I_{\rm res}^\gamma - \mathcal I_{\rm res}^\infty  \right| \leq \frac{C}{\gamma^{\alpha}}.
\end{equation}
 
\end{proof}

\begin{remark}
    It is natural to try to use the EL equation (equation \eqref{eq:corrector_equation}) for the correctors in order to cancel the inner products (for example in equation \eqref{eq:innerprod}). However, this does not work, since the corrector equation (equation \eqref{eq:corrector_equation}) only holds for fields $d'$ that are independent of $y \in \mathcal{Y}$. 
\end{remark}

We now turn to the proof of Remark \ref{T:quant3}. The proof is very similar to that of Theorem \ref{T:quant}  item (a). Hence, instead of writing a full proof, we will just indicate the points at which the proofs differ. 

\begin{proof}[Proof of Remark \ref{T:quant3}]

Using once again $\varphi_{\infty} \ast \mu_{l} + \gamma D_{\infty} $ as a test function in the definition of $Q^{\gamma}_{\eff} (x', G)$ we get:
\begin{equation}
\label{eq:comparison2}
\begin{split}
Q^{\gamma}_{\eff} (x', G) &\leq Q^{\infty}_{\eff} (x', G) + \int_{\Box }\frac{1}{\gamma^{2}}Q\left(x', x_3,y,  \left( 0 | 0 | \partial_{3 }\varphi_{\infty}\ast \mu_{l} \right) \right) \dd (x_3,y)\\
&\ \ \ \ \ \ \ \ \ \ \ \ +  \int_{\Box } Q\left(x', x_3,y, (\nabla'\varphi_{\infty}\ast \mu_{l}-\nabla'\varphi_{\infty}|0) \right)\dd (x_3,y)\\
&\ \ \ \ \ \ \ \ \ \ \ \ + \frac{1}{\gamma} \int_{\Box }\mathbb L \left(\iota(x_3G) +\iota(M_{\infty}) + \left( \nabla'\varphi_{\infty} | d_{\infty} \right) : \left( 0 | 0 | \partial_{3}\varphi_{\infty} \ast \mu_{l}  \right) \right)\dd (x_3,y)\\
&\ \ \ \ \ \ \ \ \ \ \ \ + \int_{\Box }\mathbb L \left((\nabla'\varphi_{\infty}\ast \mu_{l}-\nabla'\varphi_{\infty}|0) : \left( 0 | 0 | \partial_{3}\varphi_{\infty} \ast \mu_{l}  \right) \right)\dd (x_3,y)
\end{split}
\end{equation}

By Cauchy-Schwartz and Lemma \ref{lem:convolution}
\begin{equation}
\label{eq:Cauchy-Sch}
\begin{split}
\frac{1}{\gamma} \left| \int_{\Box }\mathbb L \left(\iota(x_3G) +\iota(M_{\infty}) + \left( \nabla'\varphi_{\infty} | d_{\infty} \right) : \left( 0 | 0 | \partial_{3}\varphi_{\infty} \ast \mu_{l}  \right) \right)\dd (x_3,y) \right| &\leq C |G|^{2} \frac{l^{(\alpha - 1)}}{\gamma}  \\
\left|\int_{\Box }\mathbb L \left((\nabla'\varphi_{\infty}\ast \mu_{l}-\nabla'\varphi_{\infty}|0) : \left( 0 | 0 | \partial_{3}\varphi_{\infty} \ast \mu_{l}  \right) \right)\dd (x_3,y) \right| & \leq C |G|^{2} l^{2 \alpha - 1} , 
\end{split}
\end{equation}

From equation \eqref{eq:comparison2} and \eqref{eq:Cauchy-Sch} we get that 
\begin{equation}
Q^{\gamma}_{\eff} (x', G) \leq Q^{\infty}_{\eff} (x', G) + c|G|^{2} \left(\frac{l^{(\alpha - 1)}}{\gamma}  +l^{2 \alpha - 1}   \right).
\end{equation}
The optimal $l$ is given by $l=\gamma^{-\alpha}$. From this we may conclude. 
\end{proof}

\begin{remark}
\label{rem:orth}
   It is not possible to mimic the rest of the structure of the proof of Theorem \ref{T:quant} for non-orthotropic materials, even at the cost of a worse rate of convergence. Proving such a rate of convergence (if it exists) would require a different approach. 
\end{remark}
    
We now turn to Theorem \ref{T:quant2}. As mentioned earlier, we will not write the whole proof, since it is very similar to that of Theorem \ref{T:quant}. Instead, we will just mention the points at which the proofs differ. 

\begin{proof}[Proof of Theorem \ref{T:quant2}]

Let $G \in \R^{2 \times 2}_{\sym}$ and let $(M,\varphi, d)$ be the corrector associated to $\iota(x_3 G)$ at infinity. Note that, since the corrector equation is linear, and $M$ is independent of $x_{3}$,
\begin{equation}
(\varphi, d) = \sum_{i=1}^{N}  \mathbf{1}_{[p_{i}, p_{i+1})}(x_{3}) \left((\varphi_{i}^{0}, d_{i}^{0}) + x_{3}(\varphi_{i}^{1}, d_{i}^{1})  \right),
\end{equation}
where $\varphi_{i}^{0},\varphi_{i}^{1}  \in  W^{1,2}(\mathcal{Y}, \R^{3}), d_{i}^{0}, d_{i}^{1} \in \R^{3}$. 

We now proceed as in the proof of Theorem \ref{T:quant}.  Consider $\varphi \ast \mu_{l}$, where $\mu: \R \to \R^{+}$ is the standard mollifier (supported in the interval $(-1,1)$) and $\mu_{l}$ is defined, for $l\leq 1$, as
\begin{equation}
\mu_{l}(x) \colonequals \frac{1}{l} \mu \left( \frac{x}{l} \right).
\end{equation}
Note that, for any $x_{3} \in I$ such that $|x_{3} - p_{i}| \geq l$ for all $i$, $\varphi(x_{3}) = \varphi \ast \mu_{l}(x_{3})$. Furthermore, for any $x_{3}$ such that, for some $i$, $|x_{3} - p_{i}| \leq l$, $\left| \partial_{3} \varphi \ast \mu_{l} (x_{3}) \right| \leq \frac{C}{l}$. These observations imply that, for $l$ small enough,
\begin{equation}
\begin{split}
\left\| \varphi - \varphi \ast \mu_{l} \right\|_{L^{2}}^{2} &\leq C l |G|^{2} \\
\left\| \partial_{3} \varphi \ast \mu_{l} \right\|_{L^{2} }^{2} &\leq \frac{C}{l} |G|^{2} .
\end{split}
\end{equation}
The proof of items (a), (b), and (c) is then exactly analogous to the proof of the same items in Theorem \ref{T:quant}. 

Similarly, denoting by $\varphi_{B}, d_{B}, M_{B}$ the corrector associated to $B$ at infinity, there holds
\begin{equation}
\begin{split}
\left\| \varphi_{B} - \varphi_{B} \ast \mu_{l} \right\|_{L^{2}}^{2} &\leq C l\\
\left\| \partial_{3} \varphi_{B} \ast \mu_{l} \right\|_{L^{2} }^{2} &\leq \frac{C}{l}.
\end{split}
\end{equation}
The proof of item (d) is then analogous to the proof of the same item in Theorem \ref{T:quant}. 

\end{proof}

Lastly, we show the proof of Proposition \ref{prop:example}.

\begin{proof}[Proof of Proposition \ref{prop:example}]

    \step 1 Identification of correctors.\\

    Let $\{G_{i}\}_{i=1}^{3}$ be the standard orthonormal basis $\R^{2 \times 2}_{\sym}$, and let $(M_{\gamma}^{i}, \varphi_{\gamma}^{i})$ denote the corrector associated to $G_{i}$. The correctors associated to $G_{1}$ and $G_{2}$ were identified in \cite[Subsection 7.7]{bohnlein2023homogenized} and they satisfy that $M_{\gamma}^{i} = 0$ and $\sym\left( \nabla_{\gamma} \varphi_{\gamma}^{i} \right)$ is independent of $\gamma$ for $i\in \{1, 2\}$. We now claim that $\varphi_{\gamma}^{3} $ is independent of $\gamma$.
    
    To prove this claim, note that for $i=3$, and $\gamma < \infty$ following \cite[Subsection 7.7]{bohnlein2023homogenized}, the corrector is given by $M_{3}=0$, and $\varphi_{\gamma}^{3} = (0, w_{*}^{\gamma}, 0)$, where $w_{*}^{\gamma}$ is given by 
    \begin{equation}
    \label{eq:correctorlaminate}
        w_{*}^{\gamma} := \operatorname{argmin}_{w} \int_{\left[ -\frac{1}{2},\frac{1}{2} \right]^{2}} \mu(y_{1}) \left( \left( \sqrt{2}x_{3} + \partial_{y_{1}} w \right)^{2} + \left( \frac{1}{\gamma} \partial_{3} w \right)^{2}\right)\dd x_3 \dd y_{1},
    \end{equation}
    where $w$ is minimized over functions of $y_{1},x_{3}$ that are periodic in $y_{1}$. For $i=3$ and $\gamma = \infty$ following again \cite[Subsection 7.7]{bohnlein2023homogenized} the corrector is given by $M_{i}=0$, and $\varphi_{\infty}^{3} = (0, w_{*}^{\infty}, 0)$, where $w_{*}^{\infty}$ is given by 
    \begin{equation} 
    \label{eq:wstar-inf}
        w_{*}^{\infty} := \operatorname{argmin}_{w} \int_{\left[ -\frac{1}{2},\frac{1}{2} \right]^{2}} \mu(y_{1})  \left( \sqrt{2}x_{3} + \partial_{y_{1}} w \right)^{2}\dd x_3 \dd y_{1},
    \end{equation}
    where $w$ is minimized over the same space as in equation \eqref{eq:correctorlaminate}. Note that $w_{*}^{\infty} $ satisfies for each $x_{3} \in I$ that
    \begin{equation}
        \partial_{y_{1}} \left(\mu(y_{1}) \left( \sqrt{2}x_{3} + \partial_{y_{1}} w_{*}^{\infty} \right) \right) = 0.
    \end{equation}
    In particular, $w_{*}^{\infty} $ is linear in $x_{3}$, and so $w_{*}^{\infty} $ also satisfies that
    \begin{equation}
        \partial^{2}_{3,3} w_{*}^{\infty}  =0.
    \end{equation}
    Therefore, $w_{*}^{\infty} $ also satisfies the Euler-Lagrange equation for $\gamma < \infty$:
    \begin{equation}
        \partial_{y_{1}} \left(\mu(y_{1}) \left( \sqrt{2}x_{3} + \partial_{y_{1}} w_{*}^{\infty} \right) \right) + \frac{1}{\gamma^{2}} \partial^{2}_{3,3} \left( \mu(y_{1})   w_{*}^{\infty} \right) = 0.
    \end{equation}

    Therefore, $\varphi_{\gamma}^{3}$ is independent of $\gamma$. Note that $\nabla_{\gamma}\varphi_{\gamma}^{3}$ still depends on $\gamma$. For the remainder of the proof, we will omit the dependence on $\gamma$ and write simply $\varphi^{3}$.

    \step 2 Proof of item (a).\\

    Note that it suffices to prove the statement for $G \in \{G_{1},G_{2},G_{3}\}$. If $G \in \{G_{1},G_{2}\}$ then, since $\sym\left( \nabla_{\gamma} \varphi_{\gamma}^{i} \right)$ is independent of $\gamma$, $Q^{\gamma}_{\eff} (G_{i}) = Q^{\infty}_{\eff} (G_{i})$ for $i \in \{1,2\}$. For the case $i=3$, using that the elastic law is isotropic, and the characterization of the corrector in step 1,
    \begin{equation}
    \begin{split}
        Q^{\gamma}_{\eff} (G_{3}) &= \int_{\Box }Q\left( x_3,y,\iota(x_3G_{3})+ \nabla_{\gamma} \varphi^{3}\right)\dd (x_3,y)\\
        &= Q^{\infty}_{\eff} (G_{3}) +\frac{1}{\gamma^{2}} \int_{\Box }Q\left( x_3,y,\iota \left(\partial_{3}\varphi^{3} \right)\right)\dd (x_3,y).
    \end{split}    
    \end{equation}
    From this, we may conclude. 
    
    \step 3 Proof of item (b).\\

    Using Proposition \ref{P:1} and \cite[Proposition 2.25]{bohnlein2023homogenized}, we have that
    \begin{equation*}
      \widehat B_i-\widehat B_i^{\gamma}
      =
      \int_{\Box}\mathbb L\left(\sym(\nabla'\varphi_{i} | 0) - \sym(\nabla_{\gamma} \varphi^{i}_{\gamma})  \right):B\dd (x_3,y),\qquad i=1,2,3.
    \end{equation*}
    In particular, if $B=\iota \left(\partial_{3}\varphi^{3} \right)$, then
    \begin{equation}
        \left| \widehat B_3-\widehat B_3^{\gamma} \right| = \frac{1}{\gamma} \int_{\Box }Q\left( x_3,y,\iota \left(\partial_{3}\varphi^{3} \right)\right)\dd (x_3,y).
    \end{equation}

    Using again Proposition \ref{P:1} and \cite[Proposition 2.25]{bohnlein2023homogenized} along with triangle inequality, there holds
    \begin{equation}
        \left| B_{\eff}^\infty - B_{\eff}^\gamma \right| \geq C \left( \left|\widehat{Q}^{-1}_{\infty} \left(\widehat B^{\gamma} - \widehat B^{\infty} \right) \right|  - \left|\left(\widehat{Q}^{-1}_{\gamma} - \widehat{Q}^{-1}_{\infty} \right)\widehat B^{\gamma} \right| \right).
    \end{equation}
    By step 2, 
    \begin{equation}
        \left|\left(\widehat{Q}^{-1}_{\gamma} - \widehat{Q}^{-1}_{\infty} \right)\widehat B^{\gamma} \right| \leq \frac{C}{\gamma^{2}},
    \end{equation}
    and so
    \begin{equation}
        \left| B_{\eff}^\infty - B_{\eff}^\gamma \right| \geq \frac{c}{\gamma} - \frac{C}{\gamma^{2}}.
    \end{equation}
    From this, we may conclude. 
\end{proof}

\appendix
\section{Proof of Lemmas \ref{L:E}, \ref{L:corrector}, and Proposition \ref{P:1}}

In this Section, we will prove results used in Section \ref{sect:Qeff}, concerning the definition of the effective quantities. The proof in the case $\gamma \in (0, \infty)$ is found in \cite{bohnlein2023homogenized} and hence will be omitted. 

We begin with the proof of Lemma \ref{L:E}. 
\begin{proof}[Proof of Lemma \ref{L:E}] 

We show the lemma for $\gamma = \infty$. The case $\gamma = 0$ follows analogously.
We will prove that, given $x' \in S$,
\begin{equation}
\label{eq:equiv'}
\alpha \int_{\square} \left| \iota \left( x_3 G \right)  \right|^{2} \leq \left\| P^{\infty, \perp}_{\rm rel} \left( \iota\left(x_{3}G\right) \right) \right\|^{2}_{x'} \leq \beta \int_{\square} \left| \iota \left( x_3 G \right)  \right|^{2},
\end{equation}
which is equivalent to equation \eqref{eq:ineq1'} in view of Definition \ref{def:class}, and also implies equation \eqref{eq:ineq2'} since $Q^{\infty}_{\hom}(x',G) = \left\| P^{\infty, \perp}_{\rm rel} \left( \iota\left(x_{3}G\right) \right) \right\|^{2}_{x'}$. The upper bound in equation \eqref{eq:equiv'} follows from the fact that $P^{\infty, \perp}_{\rm rel}$ is a projection. We will hence prove the lower bound. 

To prove the lower bound, we note that 
\begin{equation}
\left\| P^{\infty, \perp}_{\rm rel} \left( \iota\left(x_{3}G\right)\right) \right\|^{2}_{x'}  = \inf_{M, \varphi, d} \left\|  \iota\left(x_{3}G\right) +\iota(M)+\sym (\nabla' \varphi | d)\right\|^{2}_{x'},
\end{equation}
where the infimum is taken over all $M\in\R^{2\times 2}_{\sym}$, $\varphi \in L^{2}(I, W^{1,2}(\mathcal{Y}, \R^{3}))$, and $d \in L^{2}(I, \R^{3})$. 

On the other hand, for any such $(M, \varphi, d)$, there holds
\begin{equation}
\label{eq:lowbound}
\begin{split}
\left\|  \iota\left(x_{3}G\right) +\iota(M)+\sym (\nabla' \varphi | d)\right\|^{2}_{x'} &\geq \alpha \int_{\square} \left| \iota\left(x_{3}G\right) +\iota(M)+\sym (\nabla' \varphi | d) \right|^{2} \dd (x_3,y)\\
&= \alpha \int_{\square} \left| \iota\left(x_{3}G\right) \right|^{2} + |M|^{2} + \left| \sym (\nabla' \varphi | d) \right|^{2} \dd (x_3,y).
\end{split}
\end{equation}

In equation \eqref{eq:lowbound}, we have used that
\begin{equation}
\int_{\square} \iota\left(x_{3}G\right) : \left(\iota(M)+\sym (\nabla' \varphi | d) \right) \,\dd (x_3,y) = 0 = \int_{\square} \iota(M) : \sym (\nabla' \varphi | d) \,\dd (x_3,y),
\end{equation}
which is easily verified by a direct computation. From this, we may conclude. For the case $\gamma = 0$, note that also
\begin{equation}
\begin{aligned}
\int_{\square} \iota\left(x_{3}G\right) : \left(\iota(M)+\sym (\nabla' \varphi + x_3\nabla'^2\zeta | g) \right) \dd (x_3,y) &= 0, \\
\int_{\square} \iota(M) : \sym (\nabla' \varphi + x_3\nabla'^2\zeta | g) \,\dd (x_3,y)&= 0,
\end{aligned}
\end{equation}
for all $M \in \R^{2\times 2}_{\sym}$, $\varphi \in W^{1,2}(\mathcal Y,\R^3)$, $\zeta \in W^{2,2}(\mathcal{Y})$ and $g \in L^2(\Box,\R^3)$.
\end{proof}

We proceed to show Lemma \ref{L:corrector}. 
\begin{proof}[Proof of Lemma \ref{L:corrector}]
We prove the lemma for $\gamma = \infty$. The case $\gamma = 0$ follows analogously. The direct method of the Calculus of Variations implies that there exists a minimizer of
\begin{equation}
\inf_{M, \varphi, d} \left\| H+\iota(M)+\sym (\nabla' \varphi | d)\right\|_{x'}^{2}
\end{equation}
where the infimum is taken over all $M\in\R^{2\times 2}_{\sym}$, $\varphi \in L^{2}(I, W^{1,2}(\mathcal{Y}, \R^{3}))$, and $d \in L^{2}(I, \R^{3})$. By taking first variations, we can verify that the minimizer, denoted $(M_H^\infty,\varphi_H^\infty, d_H^\infty)$ satisfies equation \eqref{eq:corrector_equation}. 

Taking $H' = \iota(M_{H}^\infty)+\sym (\nabla' \varphi_{H}^\infty | d_{H}^\infty)$ as test functions in equation \eqref{eq:corrector_equation}, we obtain
\begin{equation}
\int_{\Box}Q\big(x', x_3,y, H' \big)\dd (x_3,y) = - \int_{\Box}\mathbb L H:H'\dd (x_3,y).
\end{equation}

By definition of admissibility, we have that
\begin{equation}
\label{eq:first}
\int_{\Box}Q\big(x', x_3,y, H' \big)\dd (x_3,y) \geq c \left\| H' \right\|^{2}_{L^{2}}.
\end{equation}

On the other hand, by Cauchy-Schwartz inequality, we have that 
\begin{equation}
\label{eq:second}
\left| \int_{\Box}\mathbb L H:H'\dd (x_3,y) \right| \leq C\left\| H \right\|_{L^{2}} \left\| H' \right\|_{L^{2}}.
\end{equation}

Combining equations \eqref{eq:first} and \eqref{eq:second}, we have that 
\begin{equation}
\left\| \iota(M_{H}^\infty)+\sym (\nabla' \varphi_{H}^\infty | d_{H}^\infty) \right\|_{L^{2}} = \left\| H' \right\|_{L^{2}} \leq C \left\| H \right\|_{L^{2}}.
\end{equation}

Furthermore, using orthogonality in the $L^{2}$ inner product, and Korns inequality, we have that 
\begin{equation}
\begin{split}
\left\| \iota(M_{H}^\infty)+\sym (\nabla' \varphi_{H}^\infty | d_{H}^\infty) \right\|_{L^{2}}^{2} &= \left\| \sym (\nabla' \varphi_{H}^\infty) \right\|_{L^{2}}^{2} + |M_{H}^\infty|^{2} + \|d_{H}^\infty\|_{L^{2}}^{2}\\
&\geq c\left\| \nabla' \varphi_{H}^\infty \right\|_{L^{2}}^{2} + |M_{H}^\infty|^{2} + \|d_{H}^\infty\|_{L^{2}}^{2}.
\end{split}
\end{equation}

From this, we may conclude equation \eqref{eq:corrector_apriori}. This also implies that, if $H=0$, then $\varphi_{H}^\infty = 0, M_{H}^\infty = 0,$ and $d_{H}^\infty = 0$, which implies that the solution is unique. For the case $\gamma = 0$, note that we also have the orthogonality,
\begin{equation}
\begin{aligned}
&\left\| \iota(M_{H}^0)+\sym (\nabla' \varphi_{H}^0 + x_3\nabla'^2 \zeta_H^0 | g_{H}^0) \right\|_{L^{2}}^{2} \\
=& \left\| \sym (\nabla' \varphi_{H}^0) \right\|_{L^{2}}^{2} + \|x_3\nabla'^2\zeta_H^0\|_{L^2}^2 + |M_{H}^0|^{2} + \|g_{H}\|_{L^{2}}^{2}.
\end{aligned}
\end{equation}

\end{proof}

To conclude, we show the proof of Proposition \ref{P:1}.
\begin{proof}[Proof of Proposition \ref{P:1}]
Again, we only prove the case $\gamma = \infty$ and the case $\gamma = 0$ follows analogously.

\step 1 Proof of item (a).\\

Equation \eqref{P:1:coercivityQhat} follows from equation \eqref{eq:ineq2'}. 

To prove equation \eqref{eq:matrixrep} let $(M_{\iota \left( x_3 G \right)},\varphi_{\iota \left( x_3 G \right)}, d_{\iota \left( x_3 G \right)})$ denote the corrector associated to $\iota \left( x_3 G \right)$ in the sense of Lemma \ref{L:corrector} for $\gamma = \infty$, and let $\chi_{G} \colonequals \iota(M_{\iota \left( x_3 G \right)})+\sym(\nabla'\varphi_{\iota \left( x_3 G \right)} | d_{\iota \left( x_3 G \right)})$. Similarly, let $\chi_{G_{i}} \colonequals \iota(M_{i})+\sym(\nabla'\varphi_{i} | d_{i})$. By linearity of the corrector equation, we then have that
\begin{equation}
\begin{split}
Q^{\infty}_{\eff}(x', G) &= \left\| P^{\infty, \perp}_{\rm rel} \left( \iota\left(x_{3}G\right) \right) \right\|^{2}_{x'}\\
&= \left\| \sum_{i=1}^{3} \widehat{G}_{i} \left( \iota(x_3 G_i) + \chi_{G_{i}} \right) \right\|^{2}_{x'}\\
&= \sum_{i,j=1}^{3} \left(\widehat{G}_{i} \iota(x_3 G_i) + \chi_{G_{i}}, \widehat{G}_{j} \left(\iota(x_3 G_j) + \chi_{G_{j}} \right) \right)_{x'} \\
&= \sum_{i,j=1}^{3} \left(\widehat{G}_{i}\iota(x_3 G_i) + \chi_{G_{i}}, \widehat{G}_{j} \left(\iota(x_3 G_j) \right) \right)_{x'} \\
&= \sum_{i,j=1}^{3} \widehat{Q}_{ij}^\infty \widehat{G}_{i}  \widehat{G}_{j}, 
\end{split}
\end{equation}
where the last two lines follow from the Euler-Lagrange equation for $\chi_{G_{i}}$ (equation \eqref{eq:corrector_equation}) and the definition of $\widehat{Q}^\infty$. 

\step 2 Proof of item (b).\\

Let $\widetilde{B} \colonequals \sum_{i=1}^{3} \left( (\widehat{Q}^\infty)^{-1} \widehat{B}^\infty \right)_{i} G_{i}$. Since $ \mathbf E^\infty(x', \cdot)$ is injective for a.e. $x' \in S$, it is enough to show that 
\begin{equation}
\label{eq:equiv}
 \mathbf E^\infty (x',\widetilde{B} ) = P^{\infty, \perp}_{\rm rel} \left( \sym B(x', \cdot) \right). 
\end{equation}
Since the range of $P^{\infty, \perp}_{\rm rel} \left( \sym B \right)$ is spanned by the vectors $\left\{\iota\left(x_{3}G_i\right) + \chi_{G_{i}} \right\}_{i=1}^{3}$, equation \eqref{eq:equiv} is equivalent to 
\begin{equation}
\label{eq:equiv2}
 \left(\mathbf E^\infty (x', \widetilde{B} ), \iota\left(x_{3}G_i\right) + \chi_{G_{i}} \right)_{x'} = \left( P^{\infty, \perp}_{\rm rel} \left( \sym B(x', \cdot) \right), \iota\left(x_{3}G_i\right) + \chi_{G_{i}} \right)_{x'}
\end{equation}
for each $ i \in \{1,2,3\}$. Since $P^{\infty, \perp}_{\rm rel} $ is self-adjoint and a projection, there holds
\begin{equation}
\begin{split}
\left( P^{\infty, \perp}_{\rm rel} \left( \sym B(x', \cdot) \right), \iota\left(x_{3}G_i\right) + \chi_{G_{i}} \right)_{x'} &= \left(  B (x', \cdot) , P^{\infty, \perp}_{\rm rel} \left( \iota\left(x_{3}G_i\right) + \chi_{G_{i}} \right) \right)_{x'} \\
&= \left(  B (x', \cdot) , \left( \iota\left(x_{3}G_i\right) + \chi_{G_{i}} \right) \right)_{x'} \\
&= \widehat{B}_{i}^\infty. 
\end{split}
\end{equation}

On the other hand, by linearity, $\mathbf E^\infty (x', \widetilde{B}) = \sum_{k=1}^{3} \left( (\widehat{Q}^\infty)^{-1} \widehat{B}^\infty \right)_{k} \left( \iota \left( x_3 G_k \right) + \chi_{G_{k}} \right)$. From this, we infer
\begin{equation}
\begin{split}
 \widehat{B}_i^\infty 
 &= \left(\widehat Q^\infty (\widehat Q^\infty)^{-1} \widehat B^\infty \right)_i
 = \sum_{k=1}^3 \widehat Q^\infty_{ik} \left((\widehat Q^\infty)^{-1} \widehat B^\infty \right)_k \\
 &= \left( \sum_{k=1}^3 \left((\widehat Q^\infty)^{-1} \widehat B^\infty \right)_k \left( \iota\left(x_{3}G_k\right) + \chi_{G_{k}} \right) , \iota\left(x_{3}G_i\right) + \chi_{G_{i}}\right)_{x'} \\
 &= \left(\mathbf E^\infty (x', \widetilde{B}^\infty)  , \iota\left(x_{3}G_i\right) + \chi_{G_{i}} \right)_{x'}.
\end{split}
\end{equation}
From this, we may conclude. 

\end{proof}

\section*{Acknowledgments}

 D.P.G.\ was supported in part by the Zuckerman STEM Leadership Program.
 S.N.\ and K.R.\ received support from the German Research Foundation (DFG) via the research unit FOR 3013, “Vector- and tensor-valued surface PDEs” (project number 417223351).

\bibliographystyle{alpha}
\bibliography{bibliography.bib}

\end{document}